\documentclass[11pt,a4paper]{article}
\usepackage[margin=25mm]{geometry}
\usepackage{amsmath,amssymb,amsthm,mathrsfs,mathtools,enumitem}
\usepackage[T1]{fontenc}
\usepackage[hidelinks]{hyperref}
\newtheorem{theorem}{Theorem}[section]
\newtheorem{lemma}[theorem]{Lemma}
\newtheorem{proposition}[theorem]{Proposition}
\newtheorem{corollary}[theorem]{Corollary}
\newtheorem{remark}[theorem]{Remark}
\newtheorem{example}[theorem]{Example}

\title{Reflected BSDEs under random horizons:
	Sharp support criteria, irregular reflection, and Cox reconstruction}
\author{
	\textbf{Youssef Ouknine}\\[3pt]
	\small Affiliate Professor, Africa Business School,\\
	\small Mohammed VI Polytechnic University (UM6P),\\
	\small Rabat Campus, Rocade Rabat--Salé, 11103, Morocco\\
	\small Member, Hassan II Academy of Sciences and Technology,\\
	\small Km 4, Avenue Mohammed VI, Rabat, Morocco\\
	\small \url{youssef.ouknine@um6p.ma}\\
	\and
	\textbf{Badr Elmansouri}\\[3pt]
	\small Cadi Ayyad University,\\
	\small National School of Applied Sciences of Marrakech,\\
	\small BP 575, Avenue Abdelkrim Khattabi,\\
	\small 40000 Guéliz, Marrakech, Morocco\\
	\small \url{b.elmansouri@uca.ac.ma}
}

\date{\today}

\begin{document}\maketitle
\begin{abstract}
We study reduction and reconstruction of reflected backward stochastic differential equations (RBSDEs) across random horizons. We establish optional and predictable contact-measure transfer criteria by dual projection, prove a Cox survival identity for arbitrary integrable reflection measures (including atoms), and give a Brownian Cox reconstruction theorem with continuous obstacles and a converse within a specified reduction class. We also isolate componentwise minimality criteria for irregular obstacles and predictable reflection. The unresolved semimartingale and obstacle transport problems for general non-Cox invariance times are formulated as open questions rather than asserted as theorems. The paper distinguishes established results from conditional constructions.
\end{abstract}
\tableofcontents

\section{Introduction}
\label{sec:introduction}

Backward stochastic differential equations (BSDEs) have
become an important tool in stochastic analysis, nonlinear
expectations, stochastic control, and mathematical finance.
Since the seminal work of Pardoux and Peng
\cite{PardouxPeng}, the theory has developed well beyond
the classical Brownian framework, encompassing more
general filtrations, discontinuous martingales, random
terminal horizons, and nonlinear constraints.

Reflected backward stochastic differential equations
(RBSDEs) constitute a particularly significant extension.
Their classical formulation, developed by El Karoui,
Kapoudjian, Pardoux, Peng and Quenez~\cite{ElKaroui},
provides a probabilistic approach to optimal stopping
and obstacle problems. In a Brownian filtration, a
reflected solution $(Y,Z,K)$ typically satisfies
\begin{equation*}
	\label{intro-classical}
	Y_t
	=
	\xi
	+\int_t^T f(s,Y_s,Z_s)\,ds
	+K_T-K_t
	-\int_t^T Z_s\,dW_s,
	\qquad
	Y_t\geq L_t,
\end{equation*}
together with the Skorokhod minimality condition
\[
\int_{(0,T]}(Y_t-L_t)\,dK_t=0.
\]
The increasing process $K$ enforces the obstacle
constraint, while the minimality condition ensures
that reflection acts only when the solution reaches
the obstacle. This contact property is central to the
connection between RBSDEs, Snell envelopes, and
nonlinear optimal stopping.

\paragraph{Irregular obstacles and generalized reflection.}

The classical theory commonly relies on regularity
assumptions on the obstacle and the solution.
Such assumptions can be restrictive in applications
involving discontinuous information, irregular
exercise rewards, or financial contracts subject to
default and early termination.

The work of Grigorova, Imkeller, Offen, Ouknine
and Quenez~\cite{Grigorova} develops reflected BSDEs
with obstacles that are not necessarily
right-continuous. Their approach relies on strong
optional supermartingales, Mertens decomposition,
and the Gal'chouk--Lenglart formula, and establishes
connections with optimal stopping under nonlinear
expectations.

Further contributions include the general-filtration
analysis of Baadi and Ouknine~\cite{Baadi}, the
predictable-obstacle framework of Bouhadou and
Ouknine~\cite{Bouhadou}, and the study of doubly
reflected BSDEs in the predictable setting by
Arharas, Bouhadou and Ouknine~\cite{Arharas}.
Akdim, Haddadi and Ouknine~\cite{AkdimHaddadiOuknine}
also investigate strong Snell envelopes and reflected
BSDEs with regulated trajectories.

These developments highlight the importance of
distinguishing optional and predictable reflection.
For irregular optional obstacles, the reflection
mechanism may involve a predictable increasing
process and a purely discontinuous optional
increasing process, whose minimality conditions
are formulated using different contact conventions.
Consequently, the classical continuous-obstacle
Skorokhod condition cannot simply be transferred
to arbitrary optional or predictable obstacles.

\subsection{Default times and reflected BSDEs}

A complementary extension of BSDE theory arises
when the model includes a random time $\tau$,
representing, for example, the default of a
counterparty, the failure of a system, or the
termination of a financial contract.

In credit-risk models, the reference filtration
$\mathbb F$ describes the information available
before default, while an enlarged filtration
$\mathbb G$ incorporates the observation of
$\tau$. The default time is then a
$\mathbb G$-stopping time, although it need not
be an $\mathbb F$-stopping time.

The mathematical foundations of default modeling
through enlargement of filtrations were developed
in particular by Bielecki, Jeanblanc and
Rutkowski~\cite{BieleckiJeanblancRutkowski},
and Jeanblanc and Le Cam~\cite{JeanblancLeCam}.
The conditional-density approach of El Karoui,
Jeanblanc and Jiao~\cite{ElKarouiJeanblancJiao}
and the survival-process approach of Jeanblanc
and Song~\cite{JeanblancSongSurvival} provide
further tools for studying random times and
their associated semimartingale decompositions.

The martingale representation problem in enlarged
filtrations is particularly important for backward
stochastic equations. Jeanblanc and
Song~\cite{JeanblancSongPRP} investigate the
martingale representation property under
progressive enlargement, while Aksamit,
Jeanblanc and Rutkowski~\cite{AksamitJeanblancRutkowski}
develop integral representation results for
martingales stopped at a random time.
Related enlargement formulas for multiple,
possibly non-ordered default times are established
by Jeanblanc, Li and Song~\cite{JeanblancLiSong}.

In this setting, a BSDE may be driven by both
a Brownian motion and a compensated default
martingale. The latter introduces a jump
coefficient and a compensator that must be
taken into account when comparing equations
in the reference and enlarged filtrations.

The analysis of BSDEs with default jumps and
their applications to nonlinear pricing and
optimal stopping has been developed, among
others, by Dumitrescu, Grigorova, Quenez
and Sulem~\cite{DumitrescuDefault}, and by
Dumitrescu, Quenez and Sulem~\cite{DumitrescuAmerican} (see also \cite{ArharasOuknine2024}).

More recently, Elmansouri~\cite{Elmansouri2025}
established existence and uniqueness for reflected
BSDEs with a default time and an obstacle that is
not necessarily right-continuous, under stochastic
Lipschitz assumptions. In a related study,
Elmansouri~\cite{ElmansouriJDG2026} examined BSDEs
and RBSDEs with default time and stochastic
Lipschitz drivers, together with applications
to the pricing of European and American options.

Elmansouri and El Otmani~\cite{ElmansouriElOtmaniOptimal}
investigated optimal control over split stopping
times in a defaultable framework, establishing
connections with reflected BSDEs with irregular
obstacles and nonlinear expectations.
Their work~\cite{ElmansouriElOtmaniJumps}
also extends reflected BSDEs with default times
to settings involving Brownian and jump noise,
with applications to optimal stopping and
the valuation of American options.

In addition, Elmansouri and El Otmani
\cite{ElmansouriElOtmaniPreprint} studied
well-posedness for default-time RBSDEs with
irregular barriers and their relationship
with nonlinear optimal stopping.
These contributions provide existence,
uniqueness, and control-theoretic results
for specific defaultable RBSDE formulations.

\subsection{Random horizons, optimal stopping,	and filtration reduction}

Beyond the existence theory of BSDEs and RBSDEs,
a separate line of research studies how optimal
stopping and valuation problems can be transferred
between the reference filtration and its
progressive enlargement.

Choulli and Alsheyab~\cite{ChoulliAlsheyab}
establish relationships between optimal stopping
problems under a random horizon and corresponding
reference-filtration problems. Their results
include the parametrization of enlarged-filtration
reward processes, representations of Snell
envelopes, and connections between optimal
stopping times.

Related questions for two-player stopping
problems are addressed by Elmansouri,
Marzougue and El Jamali~\cite{ElmansouriDynkin},
who examine Dynkin games under progressive
enlargement and establish corresponding
reference-filtration reduction results.

A further important framework is the theory
of invariance times developed by Cr\'epey
and Song~\cite{CrepeySong}. Under suitable
filtration and measure-change hypotheses,
this theory provides a mechanism for
transporting reference-filtration local
martingales, stopped immediately before
a random time, into the enlarged filtration.

However, martingale transport and reflection
minimality are different mathematical properties.
Even when the martingale component of a BSDE
admits a suitable transformation, this does
not automatically establish that its
finite-variation reflection measure preserves
the required contact condition.

\subsection{The reflection-support problem}

The additional difficulty can be described
directly in terms of the random measure
associated with an increasing reflection
process.

Consider an $\mathbb F$-predictable increasing
process $A$ and a nonnegative predictable
contact integrand $H$. Suppose that the
enlarged-filtration formulation provides
\begin{equation*}
	\label{intro-stopped}
	\int_{(0,T]}
	H_t\mathbf1_{\{t\leq\tau\}}\,dA_t
	=0
	\qquad\text{a.s.}
\end{equation*}

To recover the reference-filtration
minimality condition, one needs to deduce
\[
\int_{(0,T]}H_t\,dA_t=0
\qquad\text{a.s.}
\]

This implication is not automatic:
stopping at $\tau$ may remove part of
the support of the reflection measure.

Writing $q$ for the predictable survival
projection, the corresponding projection
identity takes the form
\begin{equation*}
	\label{intro-projection}
	\mathbb E\int_{(0,T]}
	H_t\mathbf1_{\{t\leq\tau\}}\,dA_t
	=
	\mathbb E\int_{(0,T]}H_tq_t\,dA_t.
\end{equation*}

The relevant positivity criterion is
therefore expressed relative to the
mean measure induced by $dA$, rather
than merely relative to Lebesgue measure.
This distinction becomes essential when
the increasing process contains atoms
or singular components.

For optional contact conditions, the
corresponding analysis involves the dual
optional projection of the stopped
reflection measure and the strict
pre-default interval. The difference
between the indicators
$\mathbf1_{\{t<\tau\}}$ and
$\mathbf1_{\{t\leq\tau\}}$ must also
be taken into account whenever the
reflection measure may charge the
graph of $\tau$.

These observations motivate a systematic
study of the support properties of
reflection measures under random horizons.

\subsection{Main contributions of the paper}

The present paper focuses on the
measure-theoretic conditions ensuring
that reflection minimality is preserved
under stopping, and on their interaction
with the reconstruction of reflected
BSDEs.

Our contributions can be summarized
as follows.

First, we establish sharp optional and
predictable contact-transfer criteria
using dual projections and the mean
measures associated with increasing
processes. The resulting conditions
identify when stopped and unstopped
contact identities are equivalent,
including when the reflection measures
possess atomic or singular components.

Second, we specialize the analysis
to Cox random times. The explicit
conditional survival structure yields
strong contact equivalence, and we
extend the survival-weighting principle
to general random measures determined
by the reference information.
Under uniform survival positivity,
we also obtain quantitative estimates
controlling approximate contact errors.

Third, in a Brownian Cox framework,
we establish an explicit reconstruction
of a reflected BSDE with continuous
obstacle from a reference-filtration
reflected equation. The transformed
generator incorporates the default
intensity and prescribed recovery,
while the enlarged solution is
reconstructed through its continuous
martingale, default-jump, and
reflection components.

The converse is obtained within an
explicitly specified reduction class,
without claiming uniqueness among
arbitrary enlarged-filtration solutions.

Fourth, for irregular obstacles,
we apply the contact-transfer criteria
separately to the predictable and
purely discontinuous optional
reflection components. We clarify
the relevant endpoint and pointwise
jump-contact conventions, while
distinguishing these minimality
statements from a complete Mertens
reconstruction theorem.

The novelty of the present analysis
thus lies in isolating the exact
support conditions required for
reflection-contact transfer and
combining them with a concrete
Cox reconstruction mechanism.
It does not consist in asserting
general existence and uniqueness
for default-time RBSDEs with
irregular obstacles, which have
already been investigated in
the aforementioned literature.

\subsection{Scope and organization}

Our results do not establish an
unconditional reduction--reconstruction
theorem for arbitrary optional obstacles
or general non-Cox invariance times.
Such an extension requires additional
hypotheses concerning martingale
representation, compensator transport,
obstacle reduction, default jumps,
and the reconstruction of Mertens
reflection components.

Section~\ref{sec:preliminaries} recalls
the probabilistic and measure-theoretic
framework. Section~\ref{sec:contact-transfer}
establishes the optional and predictable
contact-transfer criteria, and
Section~\ref{sec:cox-survival} develops
their Cox specialization, including
general reference random measures.

The Brownian Cox reconstruction theorem
is presented in
Section~\ref{sec:cox-rbsde}.
The subsequent sections investigate
irregular Mertens contact conditions,
quantitative stability, and the
remaining questions concerning
invariance-time reconstruction
beyond the Cox setting.


\section{Measure-theoretic preliminaries}
\label{sec:preliminaries}

This section introduces the probabilistic framework and
the measure-theoretic tools used throughout the paper.
Particular attention is devoted to filtration enlargement,
random Stieltjes measures, and dual optional and predictable
projections. These notions will be essential for studying
the transfer of reflection contact conditions across
random horizons.

\subsection{Probabilistic framework and filtration enlargement}

Let $(\Omega,\mathcal A,\mathbb P)$ be a complete
probability space endowed with a reference filtration
\[
\mathbb F=(\mathcal F_t)_{t\geq0}
\]
satisfying the usual conditions.

We consider a second filtration
\[
\mathbb G=(\mathcal G_t)_{t\geq0},
\]
also satisfying the usual conditions, such that
\begin{equation}\label{eq:filtration-inclusion}
	\mathcal F_t\subseteq\mathcal G_t,
	\qquad t\geq0.
\end{equation}

Let $\tau$ be a nonnegative
$\mathbb G$-stopping time, not necessarily an
$\mathbb F$-stopping time, and fix a finite
deterministic horizon $T>0$.

The inclusion \eqref{eq:filtration-inclusion}
describes a situation in which the enlarged
filtration $\mathbb G$ contains additional
information relative to the reference filtration
$\mathbb F$.

Two classical enlargement mechanisms are
particularly relevant.

\paragraph{Progressive enlargement.}

Let $\tau$ be a nonnegative random time,
not necessarily an $\mathbb F$-stopping time.

The progressive enlargement of $\mathbb F$
by $\tau$ is the smallest right-continuous
filtration containing $\mathbb F$ for which
$\tau$ becomes a stopping time.

It can be defined by
\begin{equation*}\label{eq:progressive-enlargement}
	\mathcal G_t
	=
	\bigcap_{s>t}
	\left(
	\mathcal F_s
	\vee
	\sigma(\tau\wedge s)
	\right),
	\qquad t\geq0,
\end{equation*}
with the usual completion.

In particular,
\[
\{\tau\leq t\}\in\mathcal G_t,
\qquad t\geq0,
\]
whereas this event need not belong to
$\mathcal F_t$.

Thus, the additional information concerns
the occurrence of the random time and is
revealed progressively.

\paragraph{Initial enlargement.}

Let $L$ be an $\mathcal A$-measurable random
variable representing additional information
available at time zero.

The initial enlargement of $\mathbb F$
by $L$ is defined by
\begin{equation*}\label{eq:initial-enlargement}
	\mathcal G_t
	=
	\bigcap_{s>t}
	\left(
	\mathcal F_s\vee\sigma(L)
	\right),
	\qquad t\geq0,
\end{equation*}
again with the usual completion.

In this construction, the entire random
variable $L$ is available from the beginning,
in contrast to progressive enlargement,
where information about $\tau$ is revealed
over time.

The two mechanisms may lead to different
semimartingale decompositions and martingale
representation properties.

For instance, an $\mathbb F$-Brownian motion
need not remain a Brownian motion in an
initially or progressively enlarged filtration.

Its decomposition may involve an additional
finite-variation term, often referred to
as an information drift, whenever the
corresponding semimartingale property holds.

\begin{remark}[Scope of the filtration framework]
	The measure-theoretic results developed below
	are formulated under the general inclusion
	$\mathbb F\subseteq\mathbb G$, with $\tau$
	assumed to be a $\mathbb G$-stopping time. They do not require $\mathbb G$ to coincide
	with the progressive enlargement of $\mathbb F$
	by $\tau$. In the case of an initial enlargement by $L$,
	the additional assumption that $\tau$ is
	a $\mathbb G$-stopping time must still be imposed. More specific assumptions on the enlargement
	structure will be introduced when studying
	the reduction and reconstruction of reflected
	backward stochastic differential equations.
\end{remark}

\subsection{Random Stieltjes measures}

Throughout the paper, unless otherwise stated,
an increasing process $a=(a_t)_{0\leq t\leq T}$
is assumed to be nonnegative,
$\mathbb F$-adapted, right-continuous,
and such that
\begin{equation}\label{eq:increasing-assumptions}
	a_0=0,
	\qquad
	\mathbb E[a_T]<\infty.
\end{equation}

Such a process induces a finite positive
random Stieltjes measure on $(0,T]$,
defined pathwise by
\begin{equation*}\label{Stieltjes.2}
	da(\omega,(s,t])
	=
	a_t(\omega)-a_s(\omega),
	\qquad 0\leq s<t\leq T.
\end{equation*}

For every nonnegative measurable process $H$
for which the integral is well-defined,
we use the notation
\begin{equation*}\label{eq:stieltjes-integral}
	H\cdot da
	:=
	\int_{(0,T]}H_t\,da_t.
\end{equation*}

An important feature is that the random measure
$da$ need not be absolutely continuous with
respect to Lebesgue measure $dt$.

In particular, $da$ may possess absolutely
continuous, atomic, or singular continuous
components.

The following examples illustrate these
different possibilities.

\begin{table}[htbp]
	\centering
	\renewcommand{\arraystretch}{1.6}
	\small
	\begin{tabular}{|c|c|c|}
		\hline
		\textbf{Increasing process}
		&
		\textbf{Stieltjes measure}
		&
		\textbf{Nature}
		\\
		\hline
		$a_t=t$
		&
		$dt$
		&
		Absolutely continuous
		\\
		\hline
		$a_t=\displaystyle\int_0^t\alpha_s\,ds$
		&
		$\alpha_t\,dt$
		&
		Absolutely continuous
		\\
		\hline
		$a_t=\mathbf1_{\{t\geq t_0\}}$
		&
		$\delta_{t_0}(dt)$
		&
		Atomic
		\\
		\hline
		$a_t=C(t)$
		&
		$dC(t)$
		&
		Singular continuous
		\\
		\hline
	\end{tabular}
	\caption{Examples of increasing processes and
		their associated Stieltjes measures.}
	\label{tab:stieltjes-types}
\end{table}

Here $\alpha$ is a nonnegative,
progressively measurable process satisfying
the required integrability conditions,
$t_0\in(0,T]$ is deterministic,
$\delta_{t_0}$ denotes the Dirac measure at
$t_0$, and $C$ is the classical Cantor
function on $[0,1]$, with $T=1$.

More generally, for almost every $\omega$,
the finite Stieltjes measure $da(\omega,dt)$
admits a Lebesgue decomposition
\[
da(\omega,dt)
=
da^{ac}(\omega,dt)
+
da^{s}(\omega,dt),
\]
where $da^{ac}$ is absolutely continuous
with respect to $dt$ and $da^{s}$ is singular.

The singular component may itself contain
an atomic part and a singular continuous part.

This distinction is important in the theory
of reflected backward stochastic differential
equations. Indeed, the increasing reflection process
is not generally absolutely continuous
with respect to time. Even when the reflection process is continuous,
its associated measure may concentrate on
a set of zero Lebesgue measure. When discontinuous or irregular obstacles
are considered, reflection may also involve
increasing components carrying atomic mass. Consequently, contact and minimality
conditions must be formulated with respect
to the reflection measure itself rather
than only with respect to Lebesgue measure.

\subsection{Financial interpretation}

Filtration enlargement provides a natural
framework for describing the impact of
additional information on the valuation
and management of financial risks.

In credit-risk modeling, the reference
filtration $\mathbb F$ may represent the
market information generated by observable
economic and financial risk factors.

A random time $\tau$ may represent the default
time of a counterparty or the termination
time of a financial contract.

The progressively enlarged filtration
$\mathbb G$ then incorporates the observation
of default as it occurs. A defaultable financial claim with contractual
maturity $T$ may terminate at the random
horizon
\[
T\wedge\tau.
\]

In contrast, initial enlargement may describe
an agent who knows an additional random signal
concerning future economic conditions or a
terminal market variable from time zero.

Such additional information may change
the semimartingale decomposition of price
processes and, consequently, the dynamics
of the associated backward stochastic
differential equations.

Reflected BSDEs arise naturally in the
valuation of financial contracts involving
an early exercise or termination option. The obstacle represents the exercise value
or another contractual constraint, while
the increasing reflection process ensures
that the value process respects that constraint. The reflection measure records the times
at which the constraint becomes binding. When the contract terminates at $\tau$,
only the portion of the reflection measure
associated with the relevant pre-default
interval contributes to the stopped equation. 


\begin{example}[Defaultable claim with early exercise]
	\label{ex:defaultable-claim}
	
	Consider a financial contract with maturity
	$T$ and counterparty default time $\tau$.
	
	Let $\xi$ denote its promised terminal
	payoff in the absence of default, and let
	$R_\tau$ denote the recovery payment when
	default occurs before or at maturity.
	
	The payoff at the random horizon
	$T\wedge\tau$ can be written as
	\begin{equation*}\label{eq:defaultable-payoff}
		\Xi
		=
		\xi\mathbf1_{\{\tau>T\}}
		+
		R_\tau\mathbf1_{\{\tau\leq T\}}.
	\end{equation*}
	
	If the contract also includes an early
	exercise feature, its valuation may be
	formulated through a reflected BSDE
	stopped at $T\wedge\tau$.
	
	The obstacle represents the exercise payoff,
	and the reflection process enforces the
	corresponding value constraint. The random horizon determines when the
	contract terminates, whereas the reflection
	measure determines when intervention is
	required to maintain the obstacle constraint.
\end{example}
This example motivates the study of
contact measures restricted to
$[0,\tau)$ or $[0,\tau]$.

\subsection{Dual optional and predictable projections}

We next recall the dual optional and
predictable projections of integrable
random measures relative to the reference
filtration $\mathbb F$.

Let $\mu$ be a nonnegative random measure
on $(0,T]$ satisfying the standard
measurability assumptions for dual projection
and
\[
\mathbb E[\mu((0,T])]<\infty.
\]

Its dual optional projection relative to
$\mathbb F$, denoted by $\mu^{o,\mathbb F}$,
is the $\mathbb F$-optional random measure
characterized by
\begin{equation*}\label{dualO}
	\mathbb E\left[
	\int_{(0,T]}H_t\,\mu(dt)
	\right]
	=
	\mathbb E\left[
	\int_{(0,T]}H_t\,
	\mu^{o,\mathbb F}(dt)
	\right]
\end{equation*}
for every bounded nonnegative
$\mathbb F$-optional process $H$.

Similarly, the dual predictable projection
$\mu^{p,\mathbb F}$ is characterized by
\begin{equation*}\label{dualP}
	\mathbb E\left[
	\int_{(0,T]}H_t\,\mu(dt)
	\right]
	=
	\mathbb E\left[
	\int_{(0,T]}H_t\,
	\mu^{p,\mathbb F}(dt)
	\right]
\end{equation*}
for every bounded nonnegative
$\mathbb F$-predictable process $H$.

Both identities extend to arbitrary
nonnegative admissible integrands
by monotone approximation.

It is important to distinguish these
dual projections of random measures
from the ordinary optional and predictable
projections of stochastic processes.

In particular, if
\[
A_t=\mu((0,t]),
\]
the dual optional projection of $\mu$
is not, in general, obtained by taking
the pointwise optional projection of
the cumulative process $A$.

The dual projection is characterized by
the preservation of expected integrals
against the corresponding class of
optional or predictable test processes.

\subsection{Mean measures and zero-contact conditions}

Let $\mathscr H$ denote either the
$\mathbb F$-optional sigma-field
$\mathcal O(\mathbb F)$ or the
$\mathbb F$-predictable sigma-field
$\mathcal P(\mathbb F)$, restricted to
$\Omega\times(0,T]$.

For an increasing process $a$ satisfying
\eqref{eq:increasing-assumptions},
define its mean measure on $\mathscr H$ by
\begin{equation*}\label{eq:mean-measure}
	m_a^{\mathscr H}(B)
	=
	\mathbb E\left[
	\int_{(0,T]}
	\mathbf1_B(\omega,t)\,da_t(\omega)
	\right],
	\qquad B\in\mathscr H.
\end{equation*}

This is a finite positive measure since
\[
m_a^{\mathscr H}
(\Omega\times(0,T])
=
\mathbb E[a_T]<\infty.
\]

For every nonnegative
$\mathscr H$-measurable process $H$,
we have
\begin{equation*}\label{eq:mean-integration}
	\mathbb E\left[
	\int_{(0,T]}H_t\,da_t
	\right]
	=
	\int_{\Omega\times(0,T]}
	H(\omega,t)\,
	m_a^{\mathscr H}(d\omega,dt).
\end{equation*}

Since both sides are nonnegative,
we obtain the following equivalence:
\begin{equation*}\label{eq:zero-mean-measure}
	\int_{(0,T]}H_t\,da_t=0
	\quad\mathrm{a.s.}
	\quad\Longleftrightarrow\quad
	H=0
	\quad m_a^{\mathscr H}\text{-a.e.}
\end{equation*}

Thus, a zero-contact condition can be
expressed as a null-set property of the
mean measure associated with the
increasing process.

This formulation will be particularly
useful when comparing stopped and
unstopped reflection measures.

\begin{lemma}[Zero-integral principle]
	\label{zero}
	
	Let $X$ and $Y$ be nonnegative random
	variables, possibly taking values in
	$[0,\infty]$, such that
	\[
	\mathbb E[X]=\mathbb E[Y].
	\]
	
	Then
	\[
	X=0\quad\mathrm{a.s.}
	\quad\Longleftrightarrow\quad
	Y=0\quad\mathrm{a.s.}
	\]
	
\end{lemma}

\begin{proof}
	
	Since $X$ and $Y$ are nonnegative,
	each variable vanishes almost surely
	if and only if its expectation is zero.
	
	Thus,
	\[
	X=0\quad\mathrm{a.s.}
	\quad\Longleftrightarrow\quad
	\mathbb E[X]=0.
	\]
	
	By the assumed equality of expectations,
	\[
	\mathbb E[X]=0
	\quad\Longleftrightarrow\quad
	\mathbb E[Y]=0.
	\]
	
	Finally, nonnegativity of $Y$ gives
	\[
	\mathbb E[Y]=0
	\quad\Longleftrightarrow\quad
	Y=0\quad\mathrm{a.s.}
	\]
	
	Combining these equivalences proves
	the result.
	
\end{proof}

\begin{remark}[Interpretation for contact measures]
	\label{remark-zero-contact}
	
	Let $\mu$ be a nonnegative integrable
	random measure and let $H\geq0$ be
	$\mathbb F$-optional.
	
	By the defining property of the
	dual optional projection,
	\[
	\mathbb E\int_{(0,T]}H_t\,\mu(dt)
	=
	\mathbb E\int_{(0,T]}H_t\,
	\mu^{o,\mathbb F}(dt).
	\]
	
	Consequently, Lemma~\ref{zero} yields
	\[
	\int_{(0,T]}H_t\,\mu(dt)=0
	\quad\mathrm{a.s.}
	\]
	if and only if
	\[
	\int_{(0,T]}H_t\,
	\mu^{o,\mathbb F}(dt)=0
	\quad\mathrm{a.s.}
	\]
	
	An analogous assertion holds for
	predictable integrands and dual
	predictable projections.
	
	Notice that this equivalence does not
	assert pathwise equality of the two
	integrals.
	
	It relies only on equality of their
	expectations and their nonnegativity.
	
	This observation provides the basic
	mechanism for transferring reflection
	contact conditions between random
	measures and their dual projections.
	
\end{remark}

\begin{example}[Positivity relative to the
	reflection measure]
	\label{ex:measure-support}
	
	Let $T=2$ and consider
	$
	a_t=\mathbf1_{\{t\geq1\}}.
	$
	Its associated Stieltjes measure is
	$
	da_t=\delta_1(dt).
	$
	
	Define the nonnegative deterministic process
	$q_t=\mathbf1_{\{t\neq1\}}$. Although $q$ is
	positive Lebesgue-almost everywhere,
	\[
	\int_{(0,2]}q_t\,da_t=q_1=0,
	\qquad
	\int_{(0,2]}da_t=1.
	\]
	Thus, positivity with respect to
	Lebesgue measure does not imply
	positivity with respect to the
	reflection measure.
	
	The relevant condition for transferring
	contact identities must therefore be
	formulated relative to the measure
	induced by $da$.
	
\end{example}

The preceding definitions and observations
provide the measure-theoretic foundation
for the optional and predictable contact
transfer criteria developed in the next
section.

\section{Exact transfer of optional and predictable contact}
\label{sec:contact-transfer}

We now investigate the transfer of contact conditions
between an increasing process in the reference filtration
$\mathbb F$ and its restriction to the stochastic interval
determined by a random time $\tau$.

The main difficulty is that a contact condition imposed
before $\tau$ need not determine the behavior of the
reflection measure on the entire deterministic horizon
$[0,T]$.

Throughout this section, $a$ satisfies the standing
assumptions of Section~\ref{sec:preliminaries}.

We introduce the two stopped random measures
\[
\mu_a^<(dt)
=
\mathbf1_{\{t<\tau\}}\,da_t,
\qquad
\mu_a^{\leq}(dt)
=
\mathbf1_{\{t\leq\tau\}}\,da_t.
\]

These measures differ only at the random time $\tau$.
More precisely, for every nonnegative measurable
process $H$,
\begin{equation}\label{eq:boundary-difference}
	\begin{aligned}
		&\int_{(0,T]}
		H_t\mathbf1_{\{t\leq\tau\}}\,da_t
		=\int_{(0,T]}
		H_t\mathbf1_{\{t<\tau\}}\,da_t+
		H_\tau\Delta a_\tau
		\mathbf1_{\{0<\tau\leq T\}},
	\end{aligned}
\end{equation}
where
$
\Delta a_t=a_t-a_{t-}
$ is the left-jump size at $t$.

Thus, in view of \eqref{eq:boundary-difference}, the two endpoint conventions agree whenever
$a$ is continuous, but they may differ when the
reflection measure has an atom at $\tau$.

We first examine optional contact conditions and
then turn to their predictable counterparts.

\subsection{Optional contact conditions}

Let
\begin{equation}\label{AZEMA}
G_t
=
\mathbb P(\tau>t\mid\mathcal F_t),
\qquad t\geq0,
\end{equation}
denote the Az\'ema supermartingale associated
with $\tau$ relative to $\mathbb F$.

More precisely, $G$ is the optional projection
of the process
$\mathbf1_{\{t<\tau\}}$.
We use its usual right-continuous version.

The following theorem characterizes when an
optional contact condition restricted to
$[0,\tau)$ can be transferred to $[0,T]$.

\begin{theorem}[Optional contact: exact criterion]
	\label{optional}
	
	Let $H\geq0$ be an $\mathbb F$-optional process,
	and let
	$
	\nu=(\mu_a^<)^{o,\mathbb F}
	$
	denote the dual optional projection of the
	stopped random measure $\mu_a^<$.
	
	Then
	\begin{equation}\label{optionalEq}
		\int_{(0,T]}H_t\,\mu_a^<(dt)=0
		\quad\mathrm{a.s.}
		\quad\Longleftrightarrow\quad
		\int_{(0,T]}H_t\,\nu(dt)=0
		\quad\mathrm{a.s.}
	\end{equation}
	
	Moreover, the dual optional projection satisfies
	\begin{equation*}\label{eq:optional-survival-measure}
		\nu(dt)=G_t\,da_t,
	\end{equation*}
	in the sense of optional random measures. Consequently,
	\begin{equation}\label{eq:optional-weighted-contact}
		\int_{(0,T]}
		H_t\mathbf1_{\{t<\tau\}}\,da_t=0
		\quad\mathrm{a.s.}
		\quad\Longleftrightarrow\quad
		\int_{(0,T]}H_tG_t\,da_t=0
		\quad\mathrm{a.s.}
	\end{equation}
	
	Let $m_a^o$ denote the mean measure
	associated with $da$ on the
	$\mathbb F$-optional sigma-field
	$\mathcal O(\mathbb F)$, as introduced
	in Section~\ref{sec:preliminaries}.
	
	Then the following equivalence holds
	for every nonnegative $\mathbb F$-optional
	process $H$:
	\begin{equation*}\label{eq:optional-full-contact}
		\begin{aligned}
			&\int_{(0,T]}
			H_t\mathbf1_{\{t<\tau\}}\,da_t=0
			\quad\mathrm{a.s.}
			\quad\Longleftrightarrow\quad
			\int_{(0,T]}H_t\,da_t=0
			\quad\mathrm{a.s.},
		\end{aligned}
	\end{equation*}
	if and only if
	\begin{equation*}\label{eq:optional-sharp-positivity}
		G>0
		\qquad m_a^o\text{-almost everywhere}.
	\end{equation*}
\end{theorem}

\begin{proof}
	
	We divide the proof into three steps.
	
	\medskip
	\noindent
	\textit{Step 1: Dual optional projection.}
	
	By the defining property of the dual optional
	projection, for every bounded nonnegative
	$\mathbb F$-optional process $H$,
	\[
	\mathbb E\int_{(0,T]}H_t\,\mu_a^<(dt)
	=
	\mathbb E\int_{(0,T]}H_t\,\nu(dt).
	\]
	
	For a general nonnegative optional process $H$,
	define
	\[
	H_t^{(n)}=H_t\wedge n,
	\qquad n\geq1.
	\]
	
	Since $H^{(n)}$ is bounded, nonnegative and
	optional, the preceding identity gives
	\[
	\mathbb E\int_{(0,T]}
	H_t^{(n)}\,\mu_a^<(dt)
	=
	\mathbb E\int_{(0,T]}
	H_t^{(n)}\,\nu(dt).
	\]
	
	As $n\to\infty$,
	$
	H_t^{(n)}\uparrow H_t.
	$
	The monotone convergence theorem therefore yields
	\[
	\mathbb E\int_{(0,T]}H_t\,\mu_a^<(dt)
	=
	\mathbb E\int_{(0,T]}H_t\,\nu(dt).
	\]
	
	Both integrals are nonnegative random variables.
	Consequently, Lemma~\ref{zero} implies
	\[
	\int_{(0,T]}H_t\,\mu_a^<(dt)=0
	\quad\mathrm{a.s.}
	\]
	if and only if
	\[
	\int_{(0,T]}H_t\,\nu(dt)=0
	\quad\mathrm{a.s.}
	\]
	
	This establishes \eqref{optionalEq}.
	
	\medskip
	\noindent
	\textit{Step 2: Identification of the projected measure.}
	
	Since $a$ is an $\mathbb F$-optional increasing
	process, the optional-projection integration
	theorem for integration against an adapted
	optional increasing process applies to its
	associated Stieltjes measure (with the
	optional Dol'eans measure as reference).
	
	For every bounded nonnegative
	$\mathbb F$-optional process $H$, it gives
	\[
	\begin{aligned}
		&\mathbb E\int_{(0,T]}
		H_t\mathbf1_{\{t<\tau\}}\,da_t
		=
		\mathbb E\int_{(0,T]}
		H_t\,{}^{o,\mathbb F}
		\bigl(\mathbf1_{\{t<\tau\}}\bigr)\,da_t.
	\end{aligned}
	\]
	
	By the definition of the Az\'ema
	supermartingale,
	$
	{}^{o,\mathbb F}
	\bigl(\mathbf1_{\{t<\tau\}}\bigr)
	=
	G_t.
	$
	
	Hence
	\begin{equation}\label{eq:optional-projection-identity}
		\mathbb E\int_{(0,T]}
		H_t\mathbf1_{\{t<\tau\}}\,da_t
		=
		\mathbb E\int_{(0,T]}H_tG_t\,da_t.
	\end{equation}
	
	On the other hand, the definition of $\nu$
	gives
	\[
	\mathbb E\int_{(0,T]}
	H_t\mathbf1_{\{t<\tau\}}\,da_t
	=
	\mathbb E\int_{(0,T]}H_t\,\nu(dt).
	\]
	
	Combining these identities, we obtain
	\[
	\mathbb E\int_{(0,T]}H_t\,\nu(dt)
	=
	\mathbb E\int_{(0,T]}H_tG_t\,da_t.
	\]
	
	Since this holds for every bounded nonnegative
	optional test process, uniqueness of the dual
	optional projection yields
	$
	\nu(dt)=G_t\,da_t
	$
	in the sense of optional random measures.
	
	Monotone approximation extends
	\eqref{eq:optional-projection-identity}
	to arbitrary nonnegative optional $H$.
	
	Applying Lemma~\ref{zero} once again proves
	\eqref{eq:optional-weighted-contact}.
	
	\medskip
	\noindent
	\textit{Step 3: Necessary and sufficient
		positivity condition.}
	
	Suppose first that
	\[
	G>0
	\qquad m_a^o\text{-a.e.}
	\]
	
	Assume
	\[
	\int_{(0,T]}
	H_t\mathbf1_{\{t<\tau\}}\,da_t=0
	\quad\mathrm{a.s.}
	\]
	
	By \eqref{eq:optional-projection-identity},
	\[
	\mathbb E\int_{(0,T]}H_tG_t\,da_t=0.
	\]
	
	Equivalently,
	\[
	\int_{\Omega\times(0,T]}
	HG\,dm_a^o=0.
	\]
	
	Since $HG\geq0$, it follows that
	\[
	HG=0
	\qquad m_a^o\text{-a.e.}
	\]
	
	The strict positivity of $G$ therefore implies
	\[
	H=0
	\qquad m_a^o\text{-a.e.}
	\]
	
	Consequently,
	\[
	\mathbb E\int_{(0,T]}H_t\,da_t=0,
	\]
	and hence
	\[
	\int_{(0,T]}H_t\,da_t=0
	\quad\mathrm{a.s.}
	\]
	
	The reverse implication follows immediately from
	\[
	0\leq
	\int_{(0,T]}
	H_t\mathbf1_{\{t<\tau\}}\,da_t
	\leq
	\int_{(0,T]}H_t\,da_t.
	\]
	
	To prove necessity, suppose that
	$
	m_a^o(\{G=0\})>0.
$
	Consider the nonnegative optional process
	$
	H_t=\mathbf1_{\{G_t=0\}}.
$
	Then
	$
	H_tG_t=0
	$
	everywhere, so
	$$
	\int_{(0,T]}
	H_t\mathbf1_{\{t<\tau\}}\,da_t=0
	\quad\mathrm{a.s.}
	$$
	However,
	\[
	\mathbb E\int_{(0,T]}H_t\,da_t
	=
	m_a^o(\{G=0\})>0.
	\]
	Thus the unstopped integral does not vanish
	almost surely.
	Therefore, universal equivalence of the
	stopped and unstopped contact conditions
	requires
	\[
	G>0
	\qquad m_a^o\text{-a.e.}
	\]
	
	The proof is complete.
	
\end{proof}

\begin{remark}[Support of the reflection measure]
	\label{rem:optional-support}
	
	The positivity condition
	\[
	G>0
	\qquad m_a^o\text{-a.e.}
	\]
	is a condition relative to the mean measure
	induced by $da$, not to Lebesgue measure.
	
	In particular, positivity of $G$ for
	Lebesgue-almost every time does not suffice
	when $da$ has an atomic or singular component. The criterion means that the reflection measure
	does not charge, in the mean-measure sense,
	any optional region on which the conditional
	probability of survival is zero.
	
\end{remark}

\subsection{Predictable contact conditions}

We now consider the case in which the increasing
process $a$ is $\mathbb F$-predictable.

This distinction is particularly relevant
for reflected BSDEs with irregular obstacles,
where the predictable increasing component
satisfies a contact condition involving
predictable integrands.

Let
$
q_t
=
{}^{p,\mathbb F}
\bigl(\mathbf1_{\{t\leq\tau\}}\bigr)
$
be the $\mathbb F$-predictable projection
of the non-strict survival indicator. In terms of the Az\'ema supermartingale \eqref{AZEMA},
the predictable projection admits the
representation
\[
q_t=G_{t-},
\qquad t>0,
\]
with the appropriate predictable version.

We retain the notation $q$ to emphasize
the predictable nature of the projection
and the endpoint convention.

\begin{theorem}[Predictable contact: sharp survival criterion]
	\label{predictable}
	Suppose, in addition, that $a$ is
	$\mathbb F$-predictable. Let $H\geq0$
	be an $\mathbb F$-predictable process.
	
	Then
	\begin{equation}\label{predEq}
		\mathbb E\int_{(0,T]}
		H_t\mathbf1_{\{t\leq\tau\}}\,da_t
		=
		\mathbb E\int_{(0,T]}H_tq_t\,da_t.
	\end{equation}
	
	Consequently,
	\begin{equation}\label{eq:predictable-contact-equivalence}
		\begin{aligned}
			&\int_{(0,T]}
			H_t\mathbf1_{\{t\leq\tau\}}\,da_t=0
			\quad\mathrm{a.s.}
			\quad\Longleftrightarrow\quad
			\int_{(0,T]}H_tq_t\,da_t=0
			\quad\mathrm{a.s.}
		\end{aligned}
	\end{equation}
	
	Moreover, the stopped predictable contact
	condition is equivalent to
	\[
	\int_{(0,T]}H_t\,da_t=0
	\quad\mathrm{a.s.}
	\]
	for every nonnegative $\mathbb F$-predictable
	process $H$ if and only if
	\begin{equation*}\label{eq:predictable-sharp-positivity}
		q>0
		\qquad
		m_a^p\text{-almost everywhere},
	\end{equation*}
	where $m_a^p$ is the mean measure associated
	with $da$ on $\mathcal P(\mathbb F)$,
	as introduced in
	Section~\ref{sec:preliminaries}.
	
\end{theorem}

\begin{proof}
	
	We first establish the expectation identity
	\eqref{predEq}.
	
	Since $a$ is an $\mathbb F$-predictable
	increasing process, the predictable-projection
	integration theorem gives, for every bounded
	nonnegative predictable process $H$,
	\[
	\begin{aligned}
		&\mathbb E\int_{(0,T]}
		H_t\mathbf1_{\{t\leq\tau\}}\,da_t
		=
		\mathbb E\int_{(0,T]}
		{}^{p,\mathbb F}
		\left(
		H_t\mathbf1_{\{t\leq\tau\}}
		\right)\,da_t.
	\end{aligned}
	\]
	
	Because $H$ is predictable,
	the multiplication property of predictable
	projection yields
	$
	{}^{p,\mathbb F}
	\left(
	H_t\mathbf1_{\{t\leq\tau\}}
	\right)
	=
	H_tq_t.
	$
	Therefore,
	\[
	\mathbb E\int_{(0,T]}
	H_t\mathbf1_{\{t\leq\tau\}}\,da_t
	=
	\mathbb E\int_{(0,T]}H_tq_t\,da_t.
	\]
	
	For an arbitrary nonnegative predictable $H$,
	apply this identity to
	$
	H^{(n)}=H\wedge n
	$
	and let $n\to\infty$.	The monotone convergence theorem proves
	\eqref{predEq}.
	Since both integrals in \eqref{predEq}
	are nonnegative random variables,
	Lemma~\ref{zero} gives
	\eqref{eq:predictable-contact-equivalence}.
	
	We now prove the positivity criterion.
	
	Suppose first that
	\[
	q>0
	\qquad m_a^p\text{-a.e.}
	\]
	
	If the stopped contact condition holds, then
	\[
	\mathbb E\int_{(0,T]}H_tq_t\,da_t=0.
	\]
	
	By the definition of the mean measure,
	\[
	\int_{\Omega\times(0,T]}
	Hq\,dm_a^p=0.
	\]
	
	Since $Hq\geq0$, this implies
	\[
	Hq=0
	\qquad m_a^p\text{-a.e.}
	\]
	
	Using the strict positivity of $q$,
	we conclude that
	\[
	H=0
	\qquad m_a^p\text{-a.e.}
	\]
	
	Thus
	\[
	\mathbb E\int_{(0,T]}H_t\,da_t=0,
	\]
	and consequently
	\[
	\int_{(0,T]}H_t\,da_t=0
	\quad\mathrm{a.s.}
	\]
	
	Conversely, if
	\[
	\int_{(0,T]}H_t\,da_t=0
	\quad\mathrm{a.s.},
	\]
	then, by nonnegativity,
	\[
	\int_{(0,T]}
	H_t\mathbf1_{\{t\leq\tau\}}\,da_t=0
	\quad\mathrm{a.s.}
	\]
	
	This proves sufficiency.
	
	For necessity, suppose that
	\[
	m_a^p(\{q=0\})>0.
	\]
	
	Choose
	\[
	H_t=\mathbf1_{\{q_t=0\}}.
	\]
	
	This process is nonnegative and predictable.
	Moreover,
	$
	H_tq_t=0.
	$
	
	By \eqref{predEq},
	\[
	\mathbb E\int_{(0,T]}
	H_t\mathbf1_{\{t\leq\tau\}}\,da_t=0,
	\]
	and hence
	\[
	\int_{(0,T]}
	H_t\mathbf1_{\{t\leq\tau\}}\,da_t=0
	\quad\mathrm{a.s.}
	\]
	
	On the other hand,
	\[
	\mathbb E\int_{(0,T]}H_t\,da_t
	=
	m_a^p(\{q=0\})>0.
	\]
	
	Therefore, the unstopped contact condition
	does not hold almost surely.
	
	The equivalence for every nonnegative
	predictable $H$ is thus possible only if
	\[
	q>0
	\qquad m_a^p\text{-a.e.}
	\]
	
	This completes the proof.
	
\end{proof}

\begin{remark}[Strict and non-strict survival]
	\label{rem:boundary-survival}
	
	The distinction between
	$\mathbf1_{\{t<\tau\}}$ and
	$\mathbf1_{\{t\leq\tau\}}$
	is essential when the random time $\tau$
	has atoms.
	
	Writing
	$
	G_t
	=
	\mathbb P(\tau>t\mid\mathcal F_t),
	$
	the predictable projection of the
	non-strict survival indicator is
	\[
	{}^{p,\mathbb F}
	\bigl(\mathbf1_{\{t\leq\tau\}}\bigr)
	=
	G_{t-},
	\qquad t>0.
	\]
	
	In contrast, the predictable projection of
	$\mathbf1_{\{t<\tau\}}$ need not coincide
	with $G_{t-}$.
	
	At a predictable time $S$, the difference
	between the two predictable projections
	is determined by the conditional mass
	$
	\mathbb P(\tau=S\mid\mathcal F_{S-})
	$
	on $\{S<\infty\}$.
	
	Consequently, replacing a strict survival
	indicator by its non-strict counterpart
	may change the predictable contact condition
	when the reflection measure charges the
	graph of a predictable time.
	
\end{remark}

\subsection{Examples and interpretation}

The following examples illustrate why
survival positivity and endpoint conventions
are essential.

\begin{example}[An atom at the random horizon]
	\label{ex:boundary-atom}
	
	Let $T=2$, let $\tau=1$ deterministically,
	and consider
	$
	a_t=\mathbf1_{\{t\geq1\}}.
	$
	
	The associated Stieltjes measure is
	$
	da_t=\delta_1(dt)
	$ (see Table \ref{tab:stieltjes-types}).
	
	For $H_t\equiv1$, we have
	\[
	\int_{(0,2]}
	H_t\mathbf1_{\{t<\tau\}}\,da_t=0,
	\]
	whereas
	\[
	\int_{(0,2]}
	H_t\mathbf1_{\{t\leq\tau\}}\,da_t=1.
	\]
	Thus, the two stopped random measures
	assign different masses to the graph of
	the random horizon.
	This example shows that the distinction
	between strict and non-strict stopping
	cannot be ignored for atomic reflection
	measures.
	
\end{example}

\begin{example}[Failure of contact transfer
	without survival positivity]
	\label{ex:contact-obstruction}
	
	Let $T=2$, let $\tau=1$ deterministically,
	and define
	\[
	a_t=(t-1)^+,
	\qquad
	H_t\equiv1.
	\]
	
	Then
	\[
	da_t=\mathbf1_{\{t>1\}}\,dt.
	\]
	
	Since $da$ is supported on $(1,2]$,
	we obtain
	\[
	\int_{(0,2]}
	H_t\mathbf1_{\{t<\tau\}}\,da_t=0.
	\]
	
	However,
	\[
	\int_{(0,2]}H_t\,da_t=1.
	\]
	
	The reference filtration may be taken
	to be trivial. In that case,
	\[
	G_t=\mathbf1_{\{t<1\}},
	\]
	which vanishes on the support of $da$.
	
	Thus the stopped contact condition holds,
	while the unstopped contact condition fails.
	
	The example demonstrates that the
	positivity condition must be imposed
	relative to the reflection measure.	
\end{example}
The Example \ref{ex:contact-obstruction} is intended to illustrate the
measure-theoretic obstruction only;
it is not a counterexample to any
specific invariance-time hypothesis.

\begin{remark}[Implications for reflected BSDEs]
	\label{rem:contact-rbsde}
	
	In a reflected BSDE, the increasing process
	$(a_t)$ may represent a reflection component,
	while $H$ represents a nonnegative gap
	between the solution and its obstacle.
	
	The unstopped minimality condition typically
	takes the form
	\[
	\int_{(0,T]}H_t\,da_t=0.
	\]
	
	After termination at the random horizon
	$\tau$, the corresponding condition
	may involve either
	\[
	\int_{(0,T]}
	H_t\mathbf1_{\{t<\tau\}}\,da_t=0
	\]
	or
	\[
	\int_{(0,T]}
	H_t\mathbf1_{\{t\leq\tau\}}\,da_t=0,
	\]
	depending on the reflection component
	and the adopted endpoint convention.
	
	The preceding theorems identify the exact
	survival-support conditions under which
	these stopped contact identities are
	equivalent to their unstopped counterparts.	Importantly, these results concern only
	the transfer of contact conditions.
	They do not, by themselves, establish
	the reduction or reconstruction of a
	reflected BSDE.
	Such a reconstruction additionally requires
	compatible martingale representations,
	driver transformations, terminal conditions,
	and obstacle reductions.
	
\end{remark}

\section{Cox construction: strong survival equivalence}
\label{sec:cox-survival}

We now specialize the preceding contact-transfer
criteria to a Cox construction. In this setting,
the conditional survival process is explicit
and strictly positive on the finite horizon.
This allows us to establish a strong equivalence
between stopped and unstopped contact conditions,
without imposing absolute continuity of the
underlying reflection measure.

\subsection{Cox random time and conditional survival}
\label{Boba}
Assume that the probability space carries a
nonnegative, continuous, nondecreasing,
$\mathbb F$-adapted process
$\Lambda=(\Lambda_t)_{t\geq0}$ satisfying
\[
\Lambda_0=0,
\qquad
\Lambda_T<\infty
\quad\mathrm{a.s.}
\]

Let $E$ be a unit exponential random variable ($E \sim \text{Exp}(1)$),
independent of $\mathcal F_\infty$, and define
the Cox random time
\begin{equation}\label{eq:cox-time}
	\tau
	=
	\inf\{t\geq0:\Lambda_t\geq E\},
\end{equation}
with the convention $\inf\varnothing=\infty$.

The process $\Lambda$ is called the \textit{cumulative
hazard process}. When it is absolutely continuous,
it admits a representation
\[
\Lambda_t=\int_0^t\lambda_s\,ds,
\]
where $\lambda$ is a nonnegative
$\mathbb F$-progressively measurable process,
called the\textit{ hazard-rate} or \textit{intensity process}.

Define the survival process
\begin{equation*}\label{eq:cox-survival}
	S_t=e^{-\Lambda_t},
	\qquad 0\leq t\leq T.
\end{equation*}

By continuity and monotonicity of $\Lambda$,
\[
\{\tau>t\}=\{E>\Lambda_t\}.
\]

Since $E$ is independent of
$\mathcal F_\infty$ and has the unit
exponential distribution,
\[
\begin{aligned}
	\mathbb P(\tau>t\mid\mathcal F_\infty)
	&=
	\mathbb P(E>\Lambda_t\mid\mathcal F_\infty)
	=e^{-\Lambda_t}=S_t.
\end{aligned}
\]

Furthermore, since $\Lambda_t$ is
$\mathcal F_t$-measurable, the tower property
gives
\begin{equation*}\label{eq:cox-azema}
	G_t
	=
	\mathbb P(\tau>t\mid\mathcal F_t)
	=
	S_t=e^{-\Lambda_t}.
\end{equation*}

Thus, the Az\'ema supermartingale associated
with $\tau$ is explicitly identified.

In particular, the finiteness of $\Lambda_T$
implies
\begin{equation}\label{eq:cox-positive-survival}
	S_t>0,
	\qquad 0\leq t\leq T,
	\quad\mathrm{a.s.}
\end{equation}

This strict positivity \eqref{eq:cox-positive-survival} is the key property
underlying the equivalence of contact conditions.

\begin{remark}[Financial interpretation]
	\label{rem:cox-financial}
	
	In reduced-form credit-risk modeling,
	$\tau$ may represent the default time
	of a counterparty, while $\Lambda_t$
	describes the cumulative default hazard.
	
	The process
	$
	S_t=e^{-\Lambda_t}=G_t
	$
	represents the probability of survival
	beyond time $t$, conditional on the
	reference information.
	
	In the particular case of a constant
	intensity $\lambda>0$, using \eqref{eq:cox-time}, we get
	\[
	\Lambda_t=\lambda t,
	\qquad
	\tau=\frac{E}{\lambda},
	\]
	and hence
	\[
	\mathbb P(\tau>t)=e^{-\lambda t}.
	\]
	
	Thus, $\tau$ follows an exponential
	distribution with parameter $\lambda$.
	
	The Cox construction provides a tractable
	model in which default information can be
	introduced through progressive enlargement,
	while the reference filtration retains
	the dynamics of the underlying market factors.
	
\end{remark}

\subsection{Conditional survival identity for random measures}

The following result is the central
measure-theoretic identity of this section.

Its significance is that it applies to
arbitrary integrable Stieltjes measures
associated with the reference filtration,
including measures with atoms or singular
continuous components.

\begin{lemma}[Conditional Fubini for Cox survival]
	\label{coxlemma}
	
	Let $H\geq0$ be an $\mathbb F$-optional
	process, and let $a$ satisfy the standing
	assumptions of Section~\ref{sec:preliminaries}.
	
	Then
	\begin{equation}\label{coxidentity}
		\mathbb E\int_{(0,T]}
		H_t\mathbf1_{\{t<\tau\}}\,da_t
		=
		\mathbb E\int_{(0,T]}H_tS_t\,da_t.
	\end{equation}
	
	Moreover, the same identity holds for the
	non-strict survival indicator:
	\begin{equation}\label{eq:cox-nonstrict-identity}
		\mathbb E\int_{(0,T]}
		H_t\mathbf1_{\{t\leq\tau\}}\,da_t
		=
		\mathbb E\int_{(0,T]}H_tS_t\,da_t.
	\end{equation}
	
	Consequently,
	\begin{equation}\label{eq:cox-endpoint-equivalence}
		\int_{(0,T]}
		H_t\mathbf1_{\{t<\tau\}}\,da_t
		=
		\int_{(0,T]}
		H_t\mathbf1_{\{t\leq\tau\}}\,da_t
		\quad\mathrm{a.s.}
	\end{equation}
	
\end{lemma}

\begin{proof}
	
	We first establish the strict survival identity
	\eqref{coxidentity}.
	
	Since $H$, $a$, and $\Lambda$ are
	$\mathbb F$-optional or adapted,
	their restrictions to $[0,T]$ are measurable
	with respect to the reference information
	$\mathcal F_\infty$.
	
	Conditionally on $\mathcal F_\infty$,
	the paths of $H$, $a$, and $\Lambda$ are fixed,
	whereas the exponential random variable $E$
	remains independent.
	
	By the definition of the Cox time, we have
	$
	\mathbf1_{\{t<\tau\}}
	=
	\mathbf1_{\{E>\Lambda_t\}}.
	$
	
	For every fixed $t\in[0,T]$,
	independence yields
	\[
	\begin{aligned}
		\mathbb E\left[
		\mathbf1_{\{E>\Lambda_t\}}
		\mid\mathcal F_\infty
		\right]
		&=
		\mathbb P(E>\Lambda_t\mid\mathcal F_\infty)
		=e^{-\Lambda_t}=S_t.
	\end{aligned}
	\]
	
	Because $H\geq0$ and $da$ is a positive
	random measure, conditional Tonelli's theorem
	allows us to interchange conditional expectation
	and integration with respect to the fixed
	Stieltjes measure.
	
	Consequently,
	\[
	\begin{aligned}
		\mathbb E\left[
		\left.
		\int_{(0,T]}
		H_t\mathbf1_{\{t<\tau\}}\,da_t
		\right|\mathcal F_\infty
		\right]
		=
		\int_{(0,T]}H_t
		\mathbb E\left[
		\mathbf1_{\{E>\Lambda_t\}}
		\mid\mathcal F_\infty
		\right]da_t
		&=
		\int_{(0,T]}H_te^{-\Lambda_t}\,da_t
		\\
		&\quad=
		\int_{(0,T]}H_tS_t\,da_t.
	\end{aligned}
	\]
	
	Taking expectations gives
	\[
	\mathbb E\int_{(0,T]}
	H_t\mathbf1_{\{t<\tau\}}\,da_t
	=
	\mathbb E\int_{(0,T]}H_tS_t\,da_t,
	\]
	which proves \eqref{coxidentity}.
	
	We next consider the non-strict survival
	indicator.
	
	By continuity of $\Lambda$,
	the events $\{t\leq\tau\}$ and
	$\{E\geq\Lambda_t\}$ coincide up to the
	possible equality event $\{E=\Lambda_t\}$.
	
	Since the exponential distribution has
	no atoms,
	$
	\mathbb P(E=\Lambda_t\mid\mathcal F_\infty)
	=0
	$
	for every fixed $t$.
	
	The same conditional Tonelli argument
	therefore gives
	\[
	\mathbb E\int_{(0,T]}
	H_t\mathbf1_{\{t\leq\tau\}}\,da_t
	=
	\mathbb E\int_{(0,T]}H_tS_t\,da_t,
	\]
	proving \eqref{eq:cox-nonstrict-identity}.
	
	Finally, the difference between the two
	stopped integrals is, by the boundary
	identity established in \eqref{eq:boundary-difference},
	\[
	\begin{aligned}
		&\int_{(0,T]}
		H_t\mathbf1_{\{t\leq\tau\}}\,da_t
		=	\int_{(0,T]}
		H_t\mathbf1_{\{t<\tau\}}\,da_t+
		H_\tau\Delta a_\tau
		\mathbf1_{\{0<\tau\leq T\}}.
	\end{aligned}
	\]
	
	The left-hand side is nonnegative, and
	the preceding identities show that
	its expectation is zero.
	
	It follows that
	\[
	H_\tau\Delta a_\tau
	\mathbf1_{\{0<\tau\leq T\}}
	=0
	\quad\mathrm{a.s.}
	\]
	
	Hence the strict and non-strict stopped
	integrals coincide almost surely,
	which proves \eqref{eq:cox-endpoint-equivalence}.
	
\end{proof}

\begin{remark}[Why atoms do not cause a boundary problem]
	\label{rem:cox-atoms}
	
	The preceding result does not require
	continuity of the increasing process $(a_t)$.
	
	Indeed, $(da_t)$ may have atoms at
	$\mathbb F$-stopping times or may possess
	a singular continuous component.	The essential feature is that the Cox
	threshold $E$ is independent of the
	reference information and has a
	continuous distribution.
	
	Conditionally on $\mathcal F_\infty$,
	the Stieltjes measure $da$ and the cumulative
	hazard $\Lambda$ are fixed.
		Thus, even when $(da_t)$ has atoms, the
	probability that the Cox random time
	coincides with an atom of $(da_t)$ is zero.
	More precisely,
	\[
	\Delta a_\tau
	\mathbf1_{\{0<\tau\leq T\}}
	=0
	\quad\mathrm{a.s.}
	\]
	
	This property is specific to the
	continuous-threshold Cox construction
	and need not hold for a general random time.
	
\end{remark}

\subsection{Strong equivalence of contact conditions}

We can now establish the transfer of
contact conditions without imposing
absolute continuity of the reflection
measure.

\begin{corollary}[Cox contact transfer for arbitrary
	Stieltjes reflection measures]
	\label{coxcor}
	
	Under the assumptions of
	Lemma~\ref{coxlemma}, for every nonnegative
	$\mathbb F$-optional process $H$,
	\begin{equation}\label{eq:cox-contact-equivalence}
		\begin{aligned}
			&\int_{(0,T]}
			H_t\mathbf1_{\{t<\tau\}}\,da_t=0
			\quad\mathrm{a.s.}
			\quad\Longleftrightarrow\quad
			\int_{(0,T]}H_t\,da_t=0
			\quad\mathrm{a.s.}
		\end{aligned}
	\end{equation}
	
	The same equivalence holds with
	$\mathbf1_{\{t\leq\tau\}}$ in place of
	$\mathbf1_{\{t<\tau\}}$.
	
	In particular, the result applies to
	$\mathbb F$-predictable increasing
	processes and predictable contact
	integrands, as well as to random
	Stieltjes measures with atomic or
	singular continuous components.
	
\end{corollary}

\begin{proof}
	
	By Lemma~\ref{coxlemma},
	\[
	\mathbb E\int_{(0,T]}
	H_t\mathbf1_{\{t<\tau\}}\,da_t
	=
	\mathbb E\int_{(0,T]}H_tS_t\,da_t.
	\]
	
	Since both integrals are nonnegative,
	Lemma~\ref{zero} implies
	\[
	\begin{aligned}
		&\int_{(0,T]}
		H_t\mathbf1_{\{t<\tau\}}\,da_t=0
		\quad\mathrm{a.s.}
		&\quad\Longleftrightarrow\qquad
		\int_{(0,T]}H_tS_t\,da_t=0
		\quad\mathrm{a.s.}
	\end{aligned}
	\]
	
	Now, by \eqref{eq:cox-positive-survival},
	$
	S_t=e^{-\Lambda_t}>0
	$
	for every $t\in[0,T]$, almost surely.
	Since $\Lambda$ is nonnegative and nondecreasing,
	we have
	\[
	0\leq\Lambda_t\leq\Lambda_T<\infty,
	\qquad 0\leq t\leq T,
	\]
	almost surely. Therefore,
	$
	0<e^{-\Lambda_T}\leq S_t=e^{-\Lambda_t}\leq1.
	$
	Since $H_t\geq0$ and $da_t$ is a positive
	measure, integration yields
	\[
	e^{-\Lambda_T}\int_{(0,T]}H_t\,da_t
	\leq
	\int_{(0,T]}H_tS_t\,da_t
	\leq
	\int_{(0,T]}H_t\,da_t.
	\]
	
	Since $e^{-\Lambda_T}>0$ almost surely,
	the middle integral vanishes if and only if
	the right-hand integral vanishes. Hence,
	\[
	\int_{(0,T]}H_tS_t\,da_t=0
	\quad\Longleftrightarrow\quad
	\int_{(0,T]}H_t\,da_t=0,
	\]
	almost surely.

	Combining the two equivalences proves
	\eqref{eq:cox-contact-equivalence}.
	
	The non-strict case follows from
	\eqref{eq:cox-nonstrict-identity}
	by the same argument.
	
\end{proof}

\begin{remark}[Connection with the general support criterion]
	\label{rem:cox-support-criterion}
	
	The Cox construction provides an explicit
	instance of the optional support criterion
	established in Theorem~\ref{optional}.
	
	Indeed, the Az\'ema supermartingale is
	$
	G_t=S_t=e^{-\Lambda_t}.
	$
	Since $\Lambda_T<\infty$ almost surely,
	\[
	G_t>0,
	\qquad 0\leq t\leq T.
	\]
	
	Consequently, the strict positivity
	condition
	\[
	G>0
	\qquad m_a^o\text{-a.e.}
	\]
	is automatically satisfied for every
	integrable $\mathbb F$-adapted increasing
	process $a$.
	
	Thus, the Cox construction ensures that
	stopping at $\tau$ does not destroy
	contact information carried by any
	integrable reference Stieltjes measure.
	
\end{remark}

\begin{example}[Contact transfer for an atomic measure]
	\label{ex:cox-atomic-contact}
	
	Let $T>t_0>0$, and suppose that
	\[
	\Lambda_t=\lambda t,
	\qquad \lambda>0.
	\]
	
	Then
	\[
	\tau=\frac{E}{\lambda},
	\qquad
	S_t=e^{-\lambda t}.
	\]
	
	Consider the increasing process
	$
	a_t=\mathbf1_{\{t\geq t_0\}},
	$
	whose Stieltjes measure is
	$
	da_t=\delta_{t_0}(dt).
	$
	
	For any nonnegative deterministic
	function $H$, we obtain
	\[
	\int_{(0,T]}
	H_t\mathbf1_{\{t<\tau\}}\,da_t
	=
	H_{t_0}\mathbf1_{\{t_0<\tau\}}.
	\]
	
	Taking expectations gives
	\[
	\begin{aligned}
		\mathbb E\int_{(0,T]}
		H_t\mathbf1_{\{t<\tau\}}\,da_t
		&=
		H_{t_0}\mathbb P(\tau>t_0)
		=
		H_{t_0}e^{-\lambda t_0}.
	\end{aligned}
	\]
	
	On the other hand,
	\[
	\int_{(0,T]}H_tS_t\,da_t
	=
	H_{t_0}e^{-\lambda t_0}.
	\]
	
	Therefore,
	\[
	\mathbb E\int_{(0,T]}
	H_t\mathbf1_{\{t<\tau\}}\,da_t
	=
	\mathbb E\int_{(0,T]}H_tS_t\,da_t.
	\]
	
	Since
	$
	e^{-\lambda t_0}>0,
	$
	the stopped contact condition is
	equivalent to
	$
	H_{t_0}=0,
	$
	which is precisely the unstopped
	contact condition
	\[
	\int_{(0,T]}H_t\,da_t=0.
	\]
	
	This example shows explicitly that
	the transfer result remains valid
	when the entire reflection measure
	is concentrated at a single time.
	
\end{example}

%
	The results of this section concern
	the preservation of zero-contact
	conditions under Cox stopping.
	They do not, by themselves, establish
	the reduction or reconstruction of
	a reflected backward stochastic
	differential equation.
	A complete reconstruction theorem
	additionally requires appropriate
	martingale representations, a compatible
	default compensator, a transformed
	generator, terminal conditions, and
	obstacle regularity.
	These additional issues will be addressed
	in the next section for a Brownian Cox model
	with continuous reflection.

\subsection{Extension to reference-measurable random measures}
\label{subsec:cox-general-measures}

The conditional Tonelli argument of Lemma~\ref{coxlemma}
continues to hold when the Stieltjes measure $da$ is
replaced by an arbitrary finite random measure determined
by the reference information. We record this extension
without repeating the preceding proof.

\begin{lemma}[Cox weighting beyond adapted Stieltjes measures]
\label{lem:cox-general-measures}
Let $\mu$ be a finite random measure on $(0,T]$ whose
values on Borel sets are $\mathcal F_\infty$-measurable,
and let $H\geq0$ be jointly
$\mathcal F_\infty\otimes\mathcal B((0,T])$-measurable.
Under the Cox assumptions of Section~\ref{sec:cox-survival},
\[
\mathbb E\int_{(0,T]}H_t\mathbf1_{\{t<\tau\}}\,\mu(dt)
=\mathbb E\int_{(0,T]}H_te^{-\Lambda_t}\,\mu(dt)
=\mathbb E\int_{(0,T]}H_t\mathbf1_{\{t\leq\tau\}}\,\mu(dt).
\]
The identities are valid for atomic, singular continuous,
and absolutely continuous random measures alike.
\end{lemma}

\begin{proof}
Condition on $\mathcal F_\infty$. The measure $\mu$ and
processes $H,\Lambda$ are then fixed, while $E$ retains its
unit exponential law. The conditional probabilities of
$\{E>\Lambda_t\}$ and $\{E\geq\Lambda_t\}$ are both
$e^{-\Lambda_t}$, because the exponential distribution
is atomless. Conditional Tonelli's theorem and then the
tower property give the stated identities.
\end{proof}

\begin{corollary}[Zero-contact transfer for reference measures]
\label{cor:cox-general-measure-contact}
Under the hypotheses of Lemma~\ref{lem:cox-general-measures},
\[
\int_{(0,T]}H_t\,\mu(dt)=0\quad\text{a.s.}
\]
if and only if either stopped contact integral, with
$\mathbf1_{\{t<\tau\}}$ or $\mathbf1_{\{t\leq\tau\}}$,
vanishes almost surely.
\end{corollary}

\begin{proof}
By Lemma~\ref{lem:cox-general-measures}, the stopped
integrals have the same expectation as the nonnegative
integral weighted by $e^{-\Lambda_t}$. Since
$e^{-\Lambda_T}\leq e^{-\Lambda_t}\leq1$ and
$e^{-\Lambda_T}>0$ almost surely, that weighted integral
vanishes precisely when the unweighted one does. The
conclusion follows from nonnegativity.
\end{proof}

\section{A complete Brownian Cox RBSDE theorem}
\label{sec:cox-rbsde}

We now establish a reduction--reconstruction result
for reflected backward stochastic differential
equations under a Cox random horizon.

In contrast to the preceding sections, which
concerned the transfer of contact conditions,
the present result provides an explicit
reconstruction of the full enlarged-filtration
equation, including its default martingale,
jump coefficient, and reflection process.

\subsection{Probabilistic setting and assumptions}

Assume that $\mathbb F$ is the augmented natural
filtration of a $d$-dimensional Brownian motion
$W=(W_t)_{0\leq t\leq T}$. We consider the Cox construction introduced
in Section~\ref{sec:cox-survival}, with
\[
\Lambda_t=\int_0^t\lambda_s\,ds,
\]
where $\lambda$ is nonnegative,
$\mathbb F$-progressively measurable, and
satisfies
\begin{equation}\label{eq:cox-intensity-bound}
	0\leq\lambda_t\leq\bar\lambda
\end{equation}
for some deterministic constant
$\bar\lambda<\infty$,
$dt\otimes d\mathbb P$-almost everywhere.

The random time $\tau$ is defined by
\[
\tau=\inf\{t\geq0:\Lambda_t\geq E\},
\]
where $E$ is a unit exponential random
variable independent of $\mathcal F_\infty$.

Let $\mathbb G$ be the usual progressive
enlargement of $\mathbb F$ by $\tau$.

Under the Cox construction, the immersion
property, also known as hypothesis $(H)$,
holds between $\mathbb F$ and its progressive
enlargement $\mathbb G$ by $\tau$
(see Jeanblanc and Le Cam~\cite[Section~5]{JeanblancLeCam}).

Consequently, every $\mathbb F$-local
martingale remains a $\mathbb G$-local
martingale. In particular, since $W$
is continuous and
$
\langle W^i,W^j\rangle_t=\delta_{ij}t,
$
L\'evy's characterization implies that
$W$ remains a $d$-dimensional
$\mathbb G$-Brownian motion.
Define the default indicator
$
D_t=\mathbf1_{\{\tau\leq t\}},
$
and the compensated default martingale
\begin{equation*}\label{eq:cox-default-martingale}
	M_t
	=
	D_t-\int_0^{t\wedge\tau}\lambda_s\,ds.
\end{equation*}

Then $M$ is a $\mathbb G$-martingale, and
$
\Delta M_\tau=1
\text{ on }\{\tau<\infty\}.
$
More precisely,
\begin{equation}\label{eq:cox-default-differential}
	dD_t
	=
	dM_t+
	\mathbf1_{\{t\leq\tau\}}\lambda_t\,dt,
\end{equation}
where the survival indicator is understood
in its predictable version.

We use the classical spaces
$\mathcal S^2_{\mathbb F}$,
$\mathcal H^2_{\mathbb F}$, and
$\mathcal A^2_{\mathbb F}$.

Here, $\mathcal S^2_{\mathbb F}$ denotes
the space of continuous adapted processes
$y$ such that
\[
\mathbb E\left[
\sup_{0\leq t\leq T}|y_t|^2
\right]<\infty,
\]
and $\mathcal H^2_{\mathbb F}$ denotes
the space of predictable processes $z$
such that
\[
\mathbb E\int_0^T|z_t|^2\,dt<\infty.
\]

The space $\mathcal A^2_{\mathbb F}$
consists of continuous, adapted,
nondecreasing processes $k$, starting
from zero, satisfying
\[
\mathbb E[k_T^2]<\infty.
\]

We impose the following assumptions.

\medskip
\noindent
\textbf{(H1) Terminal condition.}
The terminal random variable satisfies
$
\xi\in L^2(\mathcal F_T).
$

\medskip
\noindent
\textbf{(H2) Obstacle and recovery process.}
The processes $L$ and $R$ are continuous,
$\mathbb F$-adapted, and belong to
$\mathcal S^2_{\mathbb F}$.
Moreover,
$
L_T\leq\xi$
$\mathrm{a.s.}
$
We also assume
\[
L_t\leq R_t,
\qquad 0\leq t\leq T,\quad \mathrm{a.s.}
\]

\medskip
\noindent
\textbf{(H3) Generator.}
The function
$
f:\Omega\times[0,T]\times
\mathbb R\times\mathbb R^d\times\mathbb R
\longrightarrow\mathbb R
$
is progressively measurable with respect
to $\mathbb F$ for each fixed $(y,z,u)$.

There exists a constant $C_f>0$ such that
\[
\begin{aligned}
	&|f(t,y,z,u)-f(t,y',z',u')|
	\leq C_f
	\left(
	|y-y'|+|z-z'|+|u-u'|
	\right)
\end{aligned}
\]
for all admissible arguments,
$dt\otimes d\mathbb P$-almost everywhere.
Furthermore,
\[
\mathbb E\int_0^T
|f(t,0,0,0)|^2\,dt<\infty.
\]

\medskip
Define the reduced generator by
\begin{equation}\label{eq:cox-reduced-driver}
	F(t,y,z)
	=
	f(t,y,z,R_t-y)
	+
	\lambda_t(R_t-y).
\end{equation}

By the Lipschitz property of $f$,
the boundedness of $\lambda$ \eqref{eq:cox-intensity-bound}, and
the square integrability of $R$,
the generator $F$ is progressively
measurable, uniformly Lipschitz in
$(y,z)$, and satisfies
\[
\mathbb E\int_0^T
|F(t,0,0)|^2\,dt<\infty.
\]

Thus, the standard assumptions for
the Brownian reflected BSDE with
continuous obstacle $L$ are satisfied (see, e.g., \cite{ElKaroui}).

\subsection{Reduction and reconstruction theorem}

\begin{theorem}[Reduction and reconstruction
	with continuous reflection]
	\label{coxthm}
	
	Under assumptions \textbf{(H1)--(H3)},
	the following assertions hold.
	
	\medskip
	\noindent
	\textit{(i) Existence and uniqueness
		of the reduced RBSDE.}
	
	There exists a unique triple
	\[
	(y,z,k)\in
	\mathcal S^2_{\mathbb F}
	\times\mathcal H^2_{\mathbb F}
	\times\mathcal A^2_{\mathbb F}
	\]
	satisfying, for every $t\in[0,T]$,
	\begin{equation}\label{reduced}
		\begin{aligned}
			y_t
			={}&\xi
			+\int_t^T F(s,y_s,z_s)\,ds
			+k_T-k_t
			-\int_t^Tz_s\,dW_s,
		\end{aligned}
	\end{equation}
	together with
	\begin{equation}\label{eq:reduced-reflection}
		y_t\geq L_t,
		\qquad 0\leq t\leq T,
	\end{equation}
	and
	\begin{equation}\label{eq:reduced-skorokhod}
		\int_{(0,T]}(y_t-L_t)\,dk_t=0.
	\end{equation}
	
	\medskip
	\noindent
	\textit{(ii) Reconstruction in the
		enlarged filtration.}
	
	Define, for $0\leq t\leq T$,
	\begin{equation}\label{eq:cox-reconstruction}
		\begin{aligned}
			Y_t
			&=
			y_t\mathbf1_{\{t<\tau\}}
			+
			R_\tau\mathbf1_{\{\tau\leq t\}},
			\\
			Z_t
			&=
			z_t\mathbf1_{\{t\leq\tau\}},
			\\
			U_t
			&=
			(R_t-y_t)\mathbf1_{\{t\leq\tau\}},
			\\
			K_t
			&=
			k_{t\wedge\tau}.
		\end{aligned}
	\end{equation}
	
	The random variable $R_\tau$ in this
	definition is used only on
	$\{\tau\leq T\}$.
	
	Set
	\begin{equation}\label{eq:cox-terminal-payoff}
		\Xi
		=
		\xi\mathbf1_{\{\tau>T\}}
		+
		R_\tau\mathbf1_{\{\tau\leq T\}}.
	\end{equation}
	
	Then $(Y,Z,U,K)$ satisfies the
	$\mathbb G$-reflected BSDE stopped
	at $T\wedge\tau$:
	\begin{equation}\label{enlarged}
		\begin{aligned}
			Y_{t\wedge\tau}
			={}&
			\Xi
			+\int_{t\wedge\tau}^{T\wedge\tau}
			f(s,Y_{s-},Z_s,U_s)\,ds
			+K_{T\wedge\tau}-K_{t\wedge\tau}
			-\int_{t\wedge\tau}^{T\wedge\tau}
			Z_s\,dW_s
			-\int_{t\wedge\tau}^{T\wedge\tau}
			U_s\,dM_s.
		\end{aligned}
	\end{equation}
	
	Moreover,
	\begin{equation*}\label{eq:cox-enlarged-obstacle}
		Y_t\geq L_t
		\qquad\text{on }\{t<\tau\},
	\end{equation*}
	and
	\begin{equation*}\label{eq:cox-enlarged-contact}
		\int_{(0,T]}
		\mathbf1_{\{t\leq\tau\}}
		(Y_{t-}-L_{t-})\,dK_t
		=0
		\qquad\mathrm{a.s.}
	\end{equation*}
	
	The reconstructed processes satisfy
	the corresponding square-integrability
	requirements, including
	\[
	\mathbb E\int_0^{T\wedge\tau}
	|U_t|^2\lambda_t\,dt<\infty.
	\]
	
	\medskip
	\noindent
	\textit{(iii) Converse within the
		reduction class.}
	
	Conversely, suppose that a solution
	$(Y,Z,U,K)$ of \eqref{enlarged},
	satisfying the corresponding obstacle
	and contact conditions, admits the
	representation
	\eqref{eq:cox-reconstruction}
	for some
	\[
	(y,z,k)\in
	\mathcal S^2_{\mathbb F}
	\times\mathcal H^2_{\mathbb F}
	\times\mathcal A^2_{\mathbb F},
	\]
	without assuming in advance that
	$(y,z,k)$ solves the reduced equation.
	
	Then $(y,z,k)$ satisfies
	\eqref{reduced}--\eqref{eq:reduced-skorokhod}.
	
	Consequently, the reconstructed solution
	is unique within the specified
	$\mathbb F$-reduction class.
	
\end{theorem}


\begin{proof}
	
	We divide the proof into four steps.
	
	\medskip
	\noindent
	\textit{Step 1: Existence and uniqueness
		of the reduced equation.}
	
	By the definition
	\eqref{eq:cox-reduced-driver},
	\[
	F(t,y,z)
	=
	f(t,y,z,R_t-y)
	+
	\lambda_t(R_t-y).
	\]
	
	The assumptions on $f$, $R$ and $\lambda$
	imply that $F$ is progressively measurable
	and uniformly Lipschitz in $(y,z)$.
	
	Indeed, for any $(y,z)$ and $(y',z')$,
	\[
	\begin{aligned}
		|F(t,y,z)-F(t,y',z')|
		&\leq
		|f(t,y,z,R_t-y)
		-f(t,y',z',R_t-y')|
		+
		\lambda_t|y-y'|
		\\
		&\leq
		C_f\bigl(2|y-y'|+|z-z'|\bigr)
		+\bar\lambda|y-y'|.
	\end{aligned}
	\]
	
	Thus,
	\[
	|F(t,y,z)-F(t,y',z')|
	\leq C_F\bigl(|y-y'|+|z-z'|\bigr)
	\]
	for a deterministic constant $C_F$.
	
	Furthermore,
	\[
	\begin{aligned}
		|F(t,0,0)|
		&\leq
		|f(t,0,0,0)|
		+
		(C_f+\bar\lambda)|R_t|.
	\end{aligned}
	\]
	
	Consequently,
	\[
	\mathbb E\int_0^T
	|F(t,0,0)|^2\,dt<\infty.
	\]
	
	Since $L$ is continuous and belongs
	to $\mathcal S^2_{\mathbb F}$,
	$\xi\in L^2(\mathcal F_T)$, and
	$L_T\leq\xi$, the classical
	existence-and-uniqueness theorem
	for Brownian reflected BSDEs \cite[Theorem 5.2]{ElKaroui} applies.
	
	Therefore, there exists a unique
	solution $(y,z,k)$ satisfying
	\eqref{reduced}--\eqref{eq:reduced-skorokhod}.
	
	\medskip
	\noindent
	\textit{Step 2: Reconstruction of the
		enlarged-filtration dynamics.}
	
	Introduce
	\[
	u_t=R_t-y_t.
	\]
	
	The reduced BSDE
	\eqref{reduced} can be written in
	differential form as
	\begin{equation}\label{eq:cox-reduced-differential}
		dy_t
		=
		-\bigl[f(t,y_t,z_t,u_t)
		+\lambda_tu_t\bigr]\,dt
		-dk_t+z_t\,dW_t.
	\end{equation}
	
	Recall that
	$
	D_t=\mathbf1_{\{\tau\leq t\}}.
	$
	
	The reconstructed process $Y$ (first line in \eqref{eq:cox-reconstruction})
	can be written as
	\begin{equation*}\label{eq:cox-Y-product}
		Y_t
		=
		y_t(1-D_t)
		+
		\int_{(0,t]}R_s\,dD_s.
	\end{equation*}
	
	This representation is valid on
	$[0,T]$: the second term is zero
	before default and equals $R_\tau$
	after default.
	
	Applying integration by parts
	to $y_t(1-D_t)$ yields
	\[
	\begin{aligned}
		d\bigl(y_t(1-D_t)\bigr)
		={}&
		(1-D_{t-})\,dy_t
		-y_{t-}\,dD_t
		-d[y,D]_t.
	\end{aligned}
	\]
	
	Since $y$ is continuous and $D$
	is of finite variation, we get
	$
	[y,D]_t=0.
	$
	Moreover,
	$
	y_{t-}=y_t.
	$
	
	Therefore,
	\[
	d\bigl(y_t(1-D_t)\bigr)
	=
	(1-D_{t-})\,dy_t-y_t\,dD_t.
	\]
	
	Adding the differential of
	$\int_{(0,t]}R_s\,dD_s$ gives
	\begin{equation}\label{eq:cox-Y-differential}
		dY_t
		=
		(1-D_{t-})\,dy_t
		+
		(R_t-y_t)\,dD_t.
	\end{equation}
	
	Substituting
	\eqref{eq:cox-reduced-differential}
	into \eqref{eq:cox-Y-differential},
	we obtain
	\[
	\begin{aligned}
		dY_t
		={}&
		-(1-D_{t-})
		f(t,y_t,z_t,u_t)\,dt
		-(1-D_{t-})\lambda_tu_t\,dt
		-(1-D_{t-})\,dk_t
		\\
		&	+(1-D_{t-})z_t\,dW_t+u_t\,dD_t.
	\end{aligned}
	\]
	
	Now, by the compensated-default
	decomposition \eqref{eq:cox-default-differential}, we have
	$
	dD_t
	=
	dM_t+(1-D_{t-})\lambda_t\,dt.
	$
	Hence
	$
	u_t\,dD_t
	=
	u_t\,dM_t
	+
	(1-D_{t-})\lambda_tu_t\,dt.
	$
	The two terms involving
	$\lambda_tu_t\,dt$ cancel.
	
	Consequently,
	\[
	\begin{aligned}
		dY_t
		={}&
		-(1-D_{t-})
		f(t,y_t,z_t,u_t)\,dt
		-(1-D_{t-})\,dk_t
		+(1-D_{t-})z_t\,dW_t
		+u_t\,dM_t.
	\end{aligned}
	\]

Recall that
$
1-D_{t-}=\mathbf1_{\{t\leq\tau\}}$ for 
 $t>0.$
Thus, the reconstructed processes satisfy
\[
\begin{aligned}
	U_t
	&=u_t(1-D_{t-})
	=u_t\mathbf1_{\{t\leq\tau\}},\\
	Z_t
	&=z_t(1-D_{t-})
	=z_t\mathbf1_{\{t\leq\tau\}},\\
	K_t
	&=k_{t\wedge\tau}.
\end{aligned}
\]

Since $M$ is constant after $\tau$, we have
\[
\int_0^t u_s\,dM_s
=
\int_0^t U_s\,dM_s.
\]

Similarly, the Brownian integral is stopped:
\[
\int_0^t Z_s\,dW_s
=
\int_0^{t\wedge\tau}z_s\,dW_s.
\]

Finally, since $k$ is continuous,
\[
dK_t=(1-D_{t-})\,dk_t,
\]
and therefore
\[
K_t
=
\int_{(0,t]}(1-D_{s-})\,dk_s
=
k_{t\wedge\tau}.
\]
Hence, the three processes retain their
pre-default dynamics and are stopped at
$\tau$, without introducing any additional
jump in the reflection process $K$.

	Before default,
	\[
	Y_{t-}=y_t,
	\qquad
	Z_t=z_t,
	\qquad
	U_t=u_t.
	\]
	
	Thus, using the fact that values
	at $t=\tau$ do not affect the
	Lebesgue integral, we obtain
	\begin{equation}\label{eq:cox-enlarged-differential}
		dY_t
		=
		-\mathbf1_{\{t\leq\tau\}}
		f(t,Y_{t-},Z_t,U_t)\,dt
		-dK_t
		+Z_t\,dW_t
		+U_t\,dM_t.
	\end{equation}
	
	At the default time $\tau\leq T$,
	continuity of $y$ and $R$ gives
	\[
	Y_{\tau-}=y_\tau,
	\qquad
	Y_\tau=R_\tau.
	\]
	
	Hence
	$
	\Delta Y_\tau
	=
	R_\tau-y_\tau
	=
	u_\tau.
	$
	Since
	$
	\Delta M_\tau=1,
	$
	we also have
	$
	\Delta Y_\tau
	=
	U_\tau\Delta M_\tau.
	$
	
	The process $K$ is continuous and
	therefore contributes no jump
	at default.
	
	The reconstructed dynamics are
	thus valid both before default
	and at the default time.
	
	Integrating
	\eqref{eq:cox-enlarged-differential}
	between $t\wedge\tau$ and
	$T\wedge\tau$ yields
	\[
	\begin{aligned}
		Y_{t\wedge\tau}
		={}&
		Y_{T\wedge\tau}
		+\int_{t\wedge\tau}^{T\wedge\tau}
		f(s,Y_{s-},Z_s,U_s)\,ds
		+K_{T\wedge\tau}
		-K_{t\wedge\tau}
		\\
		&-\int_{t\wedge\tau}^{T\wedge\tau}
		Z_s\,dW_s
		-\int_{t\wedge\tau}^{T\wedge\tau}
		U_s\,dM_s.
	\end{aligned}
	\]
	
	Finally, if $\tau>T$, then
	$
	Y_T=y_T=\xi,
	$
	whereas if $\tau\leq T$, then
	$
	Y_\tau=R_\tau.
	$
	Therefore, using \eqref{eq:cox-terminal-payoff}, we derive
	\[
	Y_{T\wedge\tau}
	=
	\xi\mathbf1_{\{\tau>T\}}
	+
	R_\tau\mathbf1_{\{\tau\leq T\}}
	=\Xi.
	\]
	
	This proves the enlarged equation
	\eqref{enlarged}.
	
	\medskip
	\noindent
	\textit{Step 3: Reflection and
		square integrability.}
	
	Before default, the reconstructed
	process satisfies
	$
	Y_t=y_t.
	$
	Since the reduced solution satisfies
	$y_t\geq L_t$, we obtain
	$
	Y_t\geq L_t
	\text{ on }\{t<\tau\}.
	$
	
	Next, recall that
	$
	K_t=k_{t\wedge\tau}.
	$
	Since $k$ is continuous,
	nondecreasing, and starts from zero,
	the same properties hold for $K$.
	Moreover,
	$
	dK_t
	=
	\mathbf1_{\{t\leq\tau\}}\,dk_t.
	$
	By continuity of $y$ and $L$,
	$
	Y_{t-}-L_{t-}
	=
	y_t-L_t
	$
	at the times charged by $dK$,
	up to an irrelevant boundary
	contribution at $\tau$.
	Consequently,
	\[
	\begin{aligned}
		&\int_{(0,T]}
		\mathbf1_{\{t\leq\tau\}}
		(Y_{t-}-L_{t-})\,dK_t
		=
		\int_{(0,T]}
		\mathbf1_{\{t\leq\tau\}}
		(y_t-L_t)\,dk_t.
	\end{aligned}
	\]
	
	Since the reduced Skorokhod condition gives
	\[
	\int_{(0,T]}(y_t-L_t)\,dk_t=0
	\]
	and the integrand is nonnegative,
	we conclude that
	\[
	\int_{(0,T]}
	\mathbf1_{\{t\leq\tau\}}
	(Y_{t-}-L_{t-})\,dK_t=0
	\quad\mathrm{a.s.}
	\]
	
	We now verify square integrability.
	
	Since
	\[
	|Y_t|
	\leq
	\sup_{0\leq s\leq T}|y_s|
	+
	\sup_{0\leq s\leq T}|R_s|,
	\]
	we obtain
	\[
	\mathbb E\left[
	\sup_{0\leq t\leq T}|Y_t|^2
	\right]<\infty.
	\]
	
	Moreover,
	\[
	\mathbb E\int_0^T|Z_t|^2\,dt
	\leq
	\mathbb E\int_0^T|z_t|^2\,dt
	<\infty.
	\]
	
	For the default coefficient,
	\[
	|U_t|^2
	\leq
	2|R_t|^2+2|y_t|^2.
	\]
	
	Using the boundedness of $\lambda$,
	\[
	\begin{aligned}
		&\mathbb E\int_0^{T\wedge\tau}
		|U_t|^2\lambda_t\,dt
		\leq
		2\bar\lambda T
		\mathbb E\left[
		\sup_{0\leq t\leq T}|R_t|^2
		+
		\sup_{0\leq t\leq T}|y_t|^2
		\right]
		<\infty.
	\end{aligned}
	\]
	
	Finally,
	\[
	K_T=k_{T\wedge\tau}\leq k_T,
	\]
	so
	\[
	\mathbb E[K_T^2]
	\leq\mathbb E[k_T^2]<\infty.
	\]
	
	Thus the reconstructed processes
	satisfy the required integrability
	and reflection conditions.
	
	\medskip
	\noindent
	\textit{Step 4: Converse and uniqueness
		within the reduction class.}
	
	Suppose conversely that a solution
	$(Y,Z,U,K)$ of \eqref{enlarged}
	admits the reduction form
	\eqref{eq:cox-reconstruction}
	for some
	\[
	(y,z,k)\in
	\mathcal S^2_{\mathbb F}
	\times\mathcal H^2_{\mathbb F}
	\times\mathcal A^2_{\mathbb F}.
	\]
	
	On the event $\{\tau>T\}$,
	we have, for every $t\in[0,T]$,
	\[
	Y_t=y_t,
	\qquad
	Z_t=z_t,
	\qquad
	U_t=R_t-y_t,
	\qquad
	K_t=k_t.
	\]
	
	Moreover, on this event,
	\[
	M_t=-\int_0^t\lambda_s\,ds.
	\]
	
	Substituting these identities into
	the enlarged equation yields
	\[
	\begin{aligned}
		y_t
		={}&\xi
		+\int_t^T
		f(s,y_s,z_s,R_s-y_s)\,ds
		+\int_t^T
		\lambda_s(R_s-y_s)\,ds
		+k_T-k_t
		-\int_t^Tz_s\,dW_s.
	\end{aligned}
	\]
	
	By the definition of $F$,
	\[
	y_t
	=
	\xi+\int_t^T F(s,y_s,z_s)\,ds
	+k_T-k_t
	-\int_t^Tz_s\,dW_s \quad \text{ on }\,\, \{\tau>T\}.
	\]

	Define the 
	continuous residual process $B$, with $\mathcal{F}_T$-measurable coordinates
	\[
	\begin{aligned}
		B_t
		:={}&y_t-\xi
		-\int_t^T F(s,y_s,z_s)\,ds
		-(k_T-k_t)
		+\int_t^Tz_s\,dW_s.
	\end{aligned}
	\]
	
	Then
	$
	B_t=0
	\text{ for every }t\in[0,T]
	$
	on $\{\tau>T\}$.\footnote{At this stage of the converse proof, the reduced
			BSDE has only been established on the survival event
			$\{\tau>T\}$. In particular, setting $t=T$ yields
			$y_T=\xi$ on this event. }
	
	Let
	\[
	V
	=
	|y_T-\xi|
	+
	\sup_{t\in\mathbb Q\cap[0,T]}|B_t|,
	\]
	with $T$ also included in the
	supremum if necessary.
	
	By continuity of $B$,
	the variable $V$ detects any
	failure of the reduced equation.
	
	Furthermore,
	$
	V\geq0,$
	$
	V\in\mathcal F_T,
	$
	and
	$
	V\mathbf1_{\{\tau>T\}}=0$
	{ a.s.}

	For every integer $n\geq1$,
	\[
	\mathbb E\left[
	(V\wedge n)\mathbf1_{\{\tau>T\}}
	\right]=0.
	\]
	
	By the Cox survival identity,
	\[
	\mathbb P(\tau>T\mid\mathcal F_T)
	=
	S_T=e^{-\Lambda_T}.
	\]
	
	Therefore,
	\[
	\begin{aligned}
		0
		&=
		\mathbb E\left[
		(V\wedge n)\mathbf1_{\{\tau>T\}}
		\right]
		=
		\mathbb E\left[
		(V\wedge n)
		\mathbb E[
		\mathbf1_{\{\tau>T\}}
		\mid\mathcal F_T]
		\right]
		=
		\mathbb E[(V\wedge n)S_T].
	\end{aligned}
	\]
	
	Since
	$
	S_T>0
	\,\mathrm{a.s.},
	$
	and $V\wedge n\geq0$, we obtain
	$
	V\wedge n=0
	\,\mathrm{a.s.}
	$
	Letting $n\to\infty$ gives
	\[
	V=0
	\qquad\mathrm{a.s.}
	\]
	
	Consequently, the reduced equation
	\eqref{reduced} holds for every
	$t\in[0,T]$, almost surely.
	
	We now recover the obstacle condition.
	
	On $\{\tau>T\}$, the enlarged obstacle
	condition implies
	\[
	y_t\geq L_t
	\qquad 0\leq t\leq T.
	\]
	
	Define
	\[
	V_L
	=
	\sup_{0\leq t\leq T}(L_t-y_t)^+.
	\]
	
	By continuity of $L$ and $y$,
	$V_L$ is $\mathcal F_T$-measurable
	and nonnegative.
	
	Moreover,
	\[
	V_L\mathbf1_{\{\tau>T\}}=0
	\quad\mathrm{a.s.}
	\]
	
	Repeating the preceding conditional
	survival argument yields
	\[
	V_L=0
	\qquad\mathrm{a.s.}
	\]
	
	Hence
	\[
	y_t\geq L_t,
	\qquad 0\leq t\leq T,
	\quad\mathrm{a.s.}
	\]
	
	Finally, the enlarged Skorokhod
	condition implies
	\[
	\int_{(0,T]}
	\mathbf1_{\{t\leq\tau\}}
	(y_t-L_t)\,dk_t=0
	\quad\mathrm{a.s.}
	\]
	
	By Corollary~\ref{coxcor},
	the Cox survival positivity implies
	\[
	\int_{(0,T]}(y_t-L_t)\,dk_t=0
	\quad\mathrm{a.s.}
	\]
	
	Thus $(y,z,k)$ satisfies the
	reduced reflected BSDE together
	with its obstacle and minimality
	conditions.
	
	Uniqueness of the reduced Brownian
	RBSDE implies uniqueness of
	$(y,z,k)$.
	
	Since the enlarged solution
	$(Y,Z,U,K)$ is determined by
	\eqref{eq:cox-reconstruction},
	uniqueness follows within the
	specified reduction class.
	
	This completes the proof.
	
\end{proof}

\begin{remark}[Meaning and limitations of the converse]
	\label{rem:cox-converse}
	
	The reduction-form hypothesis in the
	converse is essential.
	
	The theorem does not assert that
	every solution of the enlarged-filtration
	RBSDE automatically admits the
	representation
	\eqref{eq:cox-reconstruction}.
	Such a statement would require a
	separate analysis of martingale
	representation, reduction of
	enlarged-filtration processes,
	and compatibility of the default
	jump with the prescribed recovery process.
	Likewise, uniqueness in the full
	class of enlarged-filtration reflected
	BSDE solutions is not established
	by the preceding argument.
	The theorem proves existence by
	explicit reconstruction and uniqueness
	within the specified reference-filtration
	reduction class.
	
\end{remark}

	Our reduction--reconstruction result is closely
	related to the framework developed by Choulli
	and Alsheyab~\cite{ChoulliAlsheyab} for optimal
	stopping under a random horizon.
	
	Their analysis considers a general random time
	$\tau$ and its progressive enlargement of
	$\mathbb F$, under the strict positivity of
	the Az\'ema supermartingale $G$.
	In particular, they establish a parametrization
	of $\mathbb G$-optional reward processes in
	terms of $\mathbb F$-observable quantities,
	derive representations of the corresponding
	Snell envelopes, and characterize the
	relationships between optimal stopping times
	in the two filtrations.
	The present theorem addresses a related but
	distinct problem. Rather than starting from
	the Snell envelope representation, we consider
	a nonlinear reflected BSDE driven by both
	a Brownian motion and a compensated default
	martingale, under a Cox random horizon.
	
	In this setting, the explicit survival process
	$
	G_t=e^{-\Lambda_t}>0
	$
	allows us to reconstruct the enlarged-filtration
	solution from the reduced Brownian reflected BSDE.
	More precisely, the default jump coefficient
	is identified by
	\[
	U_t=R_t-y_t
	\quad\text{before default},
	\]
	while the reduced generator takes the form
	\[
	F(t,y,z)
	=
	f(t,y,z,R_t-y)
	+
	\lambda_t(R_t-y).
	\]
	
	Furthermore, the survival-support identities
	established in the preceding sections ensure
	the transfer of the Skorokhod contact condition,
	including for general reference Stieltjes
	reflection measures.
	Thus, our result complements the optimal
	stopping approach of Choulli and Alsheyab \cite{ChoulliAlsheyab}
	by providing an explicit equation-level
	reconstruction in the Brownian Cox setting.

\section{Irregular obstacles and Mertens reflection}

Reflected BSDEs with irregular optional obstacles may
require two distinct increasing processes in the
reflection term. In the Mertens formulation, this
term typically takes the form
\[
A_t+C_{t-},
\]
where $A$ is predictable and increasing, while
$C$ is optional, purely discontinuous, and increasing.
See, for instance, Grigorova et
al.~\cite{Grigorova}.

For each $t\in(0,T]$, the left-upper
semicontinuous envelope of $\zeta$ is defined by
\[
{
	\bar\zeta_t
	:=
	\limsup_{\substack{s\uparrow t\\s<t}}\zeta_s
	=
	\lim_{\varepsilon\downarrow0}
	\sup_{s\in[(t-\varepsilon)^+,\,t)}\zeta_s.
}
\]
In particular, if $\zeta$ admits left limits,
then $\bar\zeta_t=\zeta_{t-}$.
This envelope allows the predictable contact
condition to be formulated even when the
obstacle has no left limits.

The two components of $A_t+C_{t-}$ correspond to different
contact mechanisms. Depending on the regularity
of the obstacle and the chosen formulation,
the minimality conditions may involve
\[
H^A_t=Y_{t-}-\bar\zeta_t,
\qquad
H^C_t=Y_t-\zeta_t,
\]
where $\bar\zeta$ denotes the appropriate
left-upper envelope of $\zeta$.

\subsection{Irregular-contact transfer}

The precise minimality conditions depend on
the regularity assumptions and on the
decomposition of the predictable reflection
component. In the following theorem, we
work directly with given nonnegative
predictable and optional contact integrands.

\begin{theorem}[Componentwise irregular-contact transfer]
	\label{mertens}
	
	Let $A$ and $C$ satisfy the standing assumptions
	on increasing processes introduced in
	Section~\ref{sec:preliminaries} 
	\footnote{Recall that these assumptions require
		the processes to be nonnegative, increasing,
		$\mathbb F$-adapted and right-continuous,
		with initial value zero and integrable terminal
		value, namely
		$A_0=C_{0-}=0$ and
		$\mathbb E[A_T]+\mathbb E[C_T]<\infty$.}.
	Assume, in addition, that $A$ is
	$\mathbb F$-predictable and $C$ is
	$\mathbb F$-optional.

	Let $H^A$ be a nonnegative
	$\mathbb F$-predictable process, and let
	$H^C$ be a nonnegative $\mathbb F$-optional
	process.
	Assume that the predictable and optional
	survival projections introduced in the
	preceding sections satisfy
	\[
	q>0\quad m_A^p\text{-a.e.},
	\qquad
	G>0\quad m_C^o\text{-a.e.}
	\]
	
	Then
	\[
	\begin{aligned}
		\int_{(0,T]}H^A_t\,dA_t=0
		\quad\text{a.s.}
		&\iff
		\int_{(0,T]}
		H^A_t\mathbf1_{\{t\leq\tau\}}\,dA_t=0
		\quad\text{a.s.},
		\\[1ex]
		\int_{(0,T]}H^C_t\,dC_t=0
		\quad\text{a.s.}
		&\iff
		\int_{(0,T]}
		H^C_t\mathbf1_{\{t<\tau\}}\,dC_t=0
		\quad\text{a.s.}
	\end{aligned}
	\]
	
	In particular, both equivalences hold
	under the Cox construction of Section~\ref{sec:cox-survival}.
	
\end{theorem}

\begin{proof}
	By Theorem~\ref{predictable}, the
	strict positivity of $q$ with respect
	to $m_A^p$ yields
	\[
	\int_{(0,T]}H^A_t\,dA_t=0
	\iff
	\int_{(0,T]}
	H^A_t\mathbf1_{\{t\leq\tau\}}\,dA_t=0
	\quad\text{a.s.}
	\]
	
	Similarly, Theorem~\ref{optional}
	and the positivity of $G$ with
	respect to $m_C^o$ give
	\[
	\int_{(0,T]}H^C_t\,dC_t=0
	\iff
	\int_{(0,T]}
	H^C_t\mathbf1_{\{t<\tau\}}\,dC_t=0
	\quad\text{a.s.}
	\]
	
	Under the Cox construction,
	$
	G_t=e^{-\Lambda_t}>0,
	$
	and the corresponding predictable
	survival projection is also strictly
	positive on the relevant interval.
	Thus, Corollary~\ref{coxcor} applies
	to each reflection component.
\end{proof}

\begin{remark}[Scope of the transfer result]
	Theorem~\ref{mertens} concerns the transfer
	of contact identities for given
	$\mathbb F$-adapted reflection components.
	
	It does not establish a reduction theorem
	for arbitrary enlarged-filtration reflected
	BSDEs with irregular obstacles.
	In particular, reducing the enlarged
	obstacle, reconstructing the finite-variation
	term $A+C_{-}$ at the default boundary,
	and identifying the corresponding martingale
	representation require separate arguments.
\end{remark}


\begin{proposition}[Absence of reflection atoms
	at a Cox default]
	\label{prop:cox-reflection-boundary}
	
	Suppose that $\tau$ is constructed as in
	Section~ \ref{sec:cox-survival}, with continuous
	cumulative intensity $\Lambda$.
	Let $C$ be an $\mathbb F$-optional,
	RCLL increasing process satisfying
	the standing assumptions of
	Section~\ref{sec:preliminaries}.
	Then
	\[
	\Delta C_\tau
	\mathbf1_{\{0<\tau\le T\}}=0
	\quad\text{a.s.}
	\]
	
	Consequently, for every nonnegative
	$\mathbb F$-optional process $H$,
	\[
	\int_{(0,T]}
	H_t\mathbf1_{\{t<\tau\}}\,dC_t
	=
	\int_{(0,T]}
	H_t\mathbf1_{\{t\le\tau\}}\,dC_t
	\quad\text{a.s.}
	\]
\end{proposition}

\begin{proof}
	Recall that
	\[
	\tau=\inf\{t\ge0:\Lambda_t\ge E\},
	\]
	where $E$ is a unit exponential random
	variable independent of $\mathcal F_\infty$.
	
	Since $\Lambda$ is continuous,
	\[
	\{\tau=t\}\subseteq\{E=\Lambda_t\},
	\qquad t<\infty.
	\]
	
	More generally, the same inclusion holds
	for any finite $\mathcal F_\infty$-measurable
	random time $\sigma$.
	
	Since $\Lambda_\sigma$ is
	$\mathcal F_\infty$-measurable and $E$
	is independent of $\mathcal F_\infty$
	with an atomless distribution,
	\[
	\mathbb P(\tau=\sigma)=0.
	\]
	
	For each $n\ge1$, the jump times of $C$
	whose jump sizes exceed $1/n$ can be
	enumerated by a finite or countable family
	of $\mathbb F$-stopping times
	$(\sigma_{n,k})_{k\ge1}$.
	
	Therefore,
	\[
	\begin{aligned}
		\mathbb P(
		0<\tau\le T,\ \Delta C_\tau\ne0)
		&\le
		\sum_{n,k\ge1}
		\mathbb P(\tau=\sigma_{n,k}<\infty)
		=0.
	\end{aligned}
	\]
	
	Finally, the difference between the two
	Stieltjes integrals is
	\[
	\begin{aligned}
		&\int_{(0,T]}
		H_t\mathbf1_{\{t\le\tau\}}\,dC_t
		-
		\int_{(0,T]}
		H_t\mathbf1_{\{t<\tau\}}\,dC_t
		=
		H_\tau\Delta C_\tau
		\mathbf1_{\{0<\tau\le T\}}
		=0
		\quad\text{a.s.}
	\end{aligned}
	\]
	This completes the proof.
\end{proof}

\begin{corollary}[Simultaneous transfer of
	Mertens contact conditions]
	\label{cor:mertens-simultaneous}
	
	Under the assumptions of
	Theorem~\ref{mertens}, the following
	equivalence holds:
	\[
	\begin{aligned}
		&\int_{(0,T]}H^A_t\,dA_t
		+
		\int_{(0,T]}H^C_t\,dC_t
		=0
		\quad\text{a.s.}\\
		&\qquad\Longleftrightarrow\\
		&\int_{(0,T]}
		H^A_t\mathbf1_{\{t\le\tau\}}\,dA_t
		+
		\int_{(0,T]}
		H^C_t\mathbf1_{\{t<\tau\}}\,dC_t
		=0
		\quad\text{a.s.}
	\end{aligned}
	\]
	
	In particular, this equivalence holds
	under the Cox construction of
	Section~\ref{sec:cox-survival}.
\end{corollary}

\begin{proof}
	All four Stieltjes integrals are
	nonnegative.
	Consequently, each sum vanishes
	almost surely if and only if both
	of its terms vanish almost surely.
	The conclusion follows by applying
	the two equivalences of
	Theorem~\ref{mertens}.
\end{proof}

\begin{example}[Failure of contact transfer
	without survival positivity]
	\label{ex:mertens-no-positivity}
	
	Let $T=2$, let $\mathbb F$ be the
	trivial filtration, and take the
	deterministic default time $\tau=1$.
	
	Define
	\[
	C_t=\mathbf1_{\{t\ge3/2\}},
	\qquad
	H^C_t=1.
	\]
	
	Then
	$
	\int_{(0,T]}H^C_t\,dC_t=1,
	$
	whereas
	$
	\int_{(0,T]}
	H^C_t\mathbf1_{\{t<\tau\}}\,dC_t=0.
	$
	Thus,
	$
	\int_{(0,T]}
	H^C_t\mathbf1_{\{t<\tau\}}\,dC_t=0
	$
	does not imply
	$
	\int_{(0,T]}H^C_t\,dC_t=0.
	$
	Indeed, the survival process satisfies
	$
	G_t=\mathbf1_{\{t<1\}},
	$
	so that $G=0$ on the support of $dC$.
	
	This example illustrates the necessity
	of the support-positivity condition
	appearing in Theorem~\ref{optional}
	and Theorem~\ref{mertens}.
\end{example}

	
	The preceding results transfer the
	predictable and optional contact
	conditions componentwise.
	In the Cox setting, Proposition~
	\ref{prop:cox-reflection-boundary}
	also shows that the strict and
	inclusive default-time stopping
	conventions coincide for the
	Stieltjes measure $dC$.


\subsection{Pointwise jump contact in Mertens reflection}
\label{subsec:mertens-jump-contact}

The componentwise transfer result of
Theorem~\ref{mertens} is formulated in terms
of nonnegative Stieltjes contact integrals.
For purely discontinuous reflection, however,
the optional minimality condition is often
expressed directly at the jump times of the
increasing process $C$.

The following proposition shows that these
two formulations are equivalent under the
optional support condition already established
in Theorem~\ref{optional}.

\begin{proposition}[Transfer of pointwise
	jump-contact minimality]
	\label{prop:pointwise-jump-contact}
	
	Let $C$ satisfy the standing assumptions of
	Section~\ref{sec:preliminaries} and suppose,
	in addition, that $C$ is purely discontinuous
	and $\mathbb F$-optional.
	
	Let $H^C$ be a nonnegative
	$\mathbb F$-optional process.
	
	Assume that
	\[
	G>0\quad m_C^o\text{-a.e.}
	\]
	
	Then the following conditions are equivalent:
	\[
	{
		H^C_t\Delta C_t=0
		\quad\text{for every }t\in(0,T],
		\quad\text{a.s.}
	}
	\]
	and
	\[
	{
		\int_{(0,T]}
		H^C_t\mathbf1_{\{t<\tau\}}\,dC_t
		=0
		\quad\text{a.s.}
	}
	\]
	
\end{proposition}

\begin{proof}
	Since $C$ is purely discontinuous and
	increasing, its Stieltjes measure is
	concentrated on its jump times. Therefore,
	for every nonnegative optional process
	$H^C$,
	\[
	\int_{(0,T]}H^C_t\,dC_t
	=
	\sum_{0<t\le T}H^C_t\Delta C_t.
	\]
	
	Every term in this sum is nonnegative.
	Consequently,
	\[
	\int_{(0,T]}H^C_t\,dC_t=0
	\quad\text{a.s.}
	\]
	if and only if
	\[
	H^C_t\Delta C_t=0
	\quad\text{for every }t\in(0,T],
	\quad\text{a.s.}
	\]
	
	On the other hand, the positivity
	assumption
	\[
	G>0\quad m_C^o\text{-a.e.}
	\]
	allows us to apply
	Theorem~\ref{optional}, yielding
	\[
	\int_{(0,T]}H^C_t\,dC_t=0
	\quad\text{a.s.}
	\]
	if and only if
	\[
	\int_{(0,T]}
	H^C_t\mathbf1_{\{t<\tau\}}\,dC_t
	=0
	\quad\text{a.s.}
	\]
	
	Combining these two equivalences
	completes the proof.
\end{proof}

\begin{corollary}[Pointwise contact for an
	irregular obstacle]
	\label{cor:mertens-pointwise-obstacle}
	
	Suppose that a compatible reduction of
	an irregular reflected BSDE has already
	identified the optional contact integrand
	as
	\[
	H^C_t=Y_t-\zeta_t\geq0,
	\]
	where $Y$ and $\zeta$ are
	$\mathbb F$-optional processes.
	
	Under the assumptions of
	Proposition~\ref{prop:pointwise-jump-contact},
	the following conditions are equivalent:
	\[
	{
		(Y_t-\zeta_t)\Delta C_t=0
		\quad\text{for every }t\in(0,T],
		\quad\text{a.s.}
	}
	\]
	and
	\[
	{
		\int_{(0,T]}
		(Y_t-\zeta_t)
		\mathbf1_{\{t<\tau\}}\,dC_t
		=0
		\quad\text{a.s.}
	}
	\]
\end{corollary}

\begin{proof}
	Apply Proposition~
	\ref{prop:pointwise-jump-contact}
	with
	\[
	H^C=Y-\zeta.
	\]
\end{proof}

\begin{remark}[Predictable versus optional
	reflection]
	\label{rem:mertens-predictable-optional}
	
	The preceding equivalence concerns only
	the purely discontinuous optional
	reflection component $C$.
	
	The predictable component $A$ is governed
	by the predictable contact condition
	involving the appropriate left-upper
	envelope of the obstacle. Its transfer
	is already covered by
	Theorem~\ref{predictable} and
	Theorem~\ref{mertens}.
	Thus, the two reflection components are
	treated separately, in accordance with
	their distinct measurability and
	minimality structures.\footnote{	In particular, the pointwise jump-contact
		equivalence does not establish the
		reconstruction of the complete Mertens
		reflection term $A+C_{-}$, nor does it
		imply a reduction--reconstruction
		theorem for arbitrary enlarged-filtration
		reflected BSDEs.}
\end{remark}

\subsection{Optional and predictable reflection:
	distinct minimality structures}
\label{subsec:optional-predictable-reflection}

The contact-transfer results established above
apply to nonnegative integrands of the appropriate
measurability class. Their application to reflected
BSDEs, however, requires a distinction between
optional and predictable obstacles.

For an irregular optional obstacle, the associated
strong optional supermartingale may admit a
Mertens decomposition involving both a
predictable increasing process $A$ and a purely
discontinuous optional increasing process $C$.
The corresponding minimality conditions are
governed by different contact mechanisms,
as discussed in Theorem~\ref{mertens} and
Proposition~\ref{prop:pointwise-jump-contact}.

In contrast, reflected BSDEs with predictable
obstacles are formulated through predictable
stopping times and require the corresponding
predictable notion of reflection. Their
minimality conditions cannot, in general,
be obtained by replacing an optional obstacle
with a predictable one in the usual
continuous-obstacle formulation.

\begin{remark}[Compatibility of contact conventions]
	\label{rem:reflection-conventions}
	
	The optional and predictable contact-transfer
	criteria concern the measures and integrands
	specified in their respective formulations.
	They do not identify the reflection processes
	of different reflected BSDE models.
	
	In particular, when a Mertens reflection term
	is written as $A+C_{-}$, the jump convention
	for $C$ must be distinguished from that of
	the process $C_{-}$.
	
	If $C$ is right-continuous, its jumps are
	$\Delta C_t=C_t-C_{t-}$, whereas the
	corresponding right jump of $C_{-}$ is
	\[
	\Delta_+(C_{-})_t
	=
	C_t-C_{t-}
	=
	\Delta C_t.
	\]
	
	Thus, the contact condition must be
	formulated using the appropriate jump
	of the increasing process appearing
	in the chosen Mertens decomposition.
	
	The transfer results of
	Theorems~\ref{optional},
	\ref{predictable}, and
	\ref{mertens} apply only after these
	contact integrands and endpoint
	conventions have been identified.
\end{remark}

\section{Quantitative reflection-support transfer
	and the limits of martingale invariance}
\label{sec:quantitative-support}

The preceding optional and predictable
contact-transfer theorems give sharp
qualitative criteria for recovering
reflection minimality from its restriction
to the pre-default interval.

We now derive a quantitative consequence
of these criteria. Unlike the equivalence
of exact contact identities, this estimate
also controls approximate reflection errors
and is therefore relevant to approximation
and penalisation procedures.

\begin{proposition}[Quantitative contact transfer]
	\label{prop:quantitative-contact}
	
	Let $(a_t)$ satisfy the standing assumptions
	of Section~\ref{sec:preliminaries}, and
	let $H$ be a nonnegative admissible
	contact integrand.
	
	Consider either of the following settings:
	
	\begin{enumerate}
		\item[(i)] $(a_t)$ and $H$ are
		$\mathbb F$-optional, and
		\[
		G_t\geq\varepsilon
		\quad m_a^o\text{-a.e.};
		\]
		
		\item[(ii)] $(a_t)$ and $H$ are
		$\mathbb F$-predictable, and
		\[
		q_t\geq\varepsilon
		\quad m_a^p\text{-a.e.},
		\]
	\end{enumerate}
	for some constant $\varepsilon>0$.
	
	Then, in the optional case,
	\[
	{
		\varepsilon\,
		\mathbb E\int_{(0,T]}H_t\,da_t
		\leq
		\mathbb E\int_{(0,T]}
		H_t\mathbf1_{\{t<\tau\}}\,da_t
		\leq
		\mathbb E\int_{(0,T]}H_t\,da_t.
	}
	\]
	
	In the predictable case,
	\[
	{
		\varepsilon\,
		\mathbb E\int_{(0,T]}H_t\,da_t
		\leq
		\mathbb E\int_{(0,T]}
		H_t\mathbf1_{\{t\leq\tau\}}\,da_t
		\leq
		\mathbb E\int_{(0,T]}H_t\,da_t.
	}
	\]
	The inequalities are understood in
	the extended nonnegative sense.
\end{proposition}

\begin{proof}
	We first consider the optional case.
	
	By the projection identity established
	in Theorem~\ref{optional},
	\[
	\mathbb E\int_{(0,T]}
	H_t\mathbf1_{\{t<\tau\}}\,da_t
	=
	\mathbb E\int_{(0,T]}H_tG_t\,da_t.
	\]
	
	Since
	\[
	\varepsilon\leq G_t\leq1
	\quad m_a^o\text{-a.e.},
	\]
	integration against the nonnegative
	function $H$ gives
	\[
	\varepsilon\,
	\mathbb E\int_{(0,T]}H_t\,da_t
	\leq
	\mathbb E\int_{(0,T]}H_tG_t\,da_t
	\leq
	\mathbb E\int_{(0,T]}H_t\,da_t.
	\]
	
	This proves the first assertion.
	
	The predictable case follows
	identically from
	Theorem~\ref{predictable},
	using the predictable projection
	$q$ in place of $G$.
\end{proof}

\begin{corollary}[Stability of approximate
	contact conditions]
	\label{cor:approximate-contact}
	
	Under either setting of
	Proposition~\ref{prop:quantitative-contact},
	let $(H^n)_{n\geq1}$ be a sequence
	of nonnegative contact integrands
	of the corresponding measurability
	class.
	
	Suppose that the constant
	$\varepsilon>0$ is independent of $n$.
	
	If the corresponding stopped
	contact errors satisfy
	\[
	\mathbb E\int_{(0,T]}
	H^n_t\mathbf1_{\{t<\tau\}}\,da_t
	\longrightarrow0
	\]
	in the optional case, or
	\[
	\mathbb E\int_{(0,T]}
	H^n_t\mathbf1_{\{t\leq\tau\}}\,da_t
	\longrightarrow0
	\]
	in the predictable case, then
	\[
	{
		\mathbb E\int_{(0,T]}H^n_t\,da_t
		\longrightarrow0.
	}
	\]
\end{corollary}

\begin{proof}
	The conclusion follows immediately
	from the lower bounds of
	Proposition~\ref{prop:quantitative-contact},
	with the same constant
	$\varepsilon$ for every $n$.
\end{proof}

\begin{remark}[Application to penalisation
	and approximation schemes]
	\label{rem:penalisation-support}
	
	Corollary~\ref{cor:approximate-contact}
	applies whenever an approximation scheme
	produces nonnegative contact errors
	and a uniform lower bound for the
	relevant survival projection.
	
	If the reflection measures also depend
	on $n$, the same argument remains valid
	provided that the lower bound holds
	uniformly with respect to their
	corresponding reference measures.
	The estimate transfers convergence of
	the contact errors in expectation.
	It does not, by itself, establish
	convergence of the approximating
	solutions or identify the limiting
	reflection measure.
\end{remark}

\begin{remark}[Obstruction to unrestricted
	invariance-time reconstruction]
	\label{rem:invariance-obstruction}
	
	The martingale-invariance property
	does not, by itself, imply a positive
	lower bound for either survival
	projection on the support of a
	reflection measure.
	
	As illustrated by
	Example~\ref{ex:mertens-no-positivity},
	a reflection measure may charge a
	region where the relevant survival
	projection vanishes. In such a case,
	the stopped contact integral may
	vanish even though the corresponding
	unstopped integral is strictly positive.
	Thus, martingale transport alone
	cannot guarantee recovery of the
	full reflection minimality condition.
	
	The qualitative support criteria
	of Theorems~\ref{optional} and
	\ref{predictable}, and the quantitative
	bounds established above, remain
	separate requirements in any
	general reduction--reconstruction
	argument.
\end{remark}

\begin{example}[Quantitative transfer for
	an atomic reflection measure]
	\label{ex:quantitative-atomic}
	
	Consider the Cox construction of
	Section~\ref{sec:cox-survival} with constant
	intensity $\lambda>0$, so that
	\[
	\Lambda_t=\lambda t,
	\qquad
	G_t=e^{-\lambda t}.
	\]
	
	Fix $t_0\in(0,T)$ and let
	$
	a_t=\mathbf1_{\{t\ge t_0\}}.
	$
	Thus, $da$ is concentrated at $t_0$.
	
	For each $n\ge1$, consider the
	nonnegative $\mathbb F$-optional
	contact error
	\[
	H^n_t=\frac1n.
	\]
	
	The full contact error is
	\[
	\mathbb E\int_{(0,T]}H^n_t\,da_t
	=\frac1n.
	\]
	
	On the other hand, the stopped
	contact error satisfies
	\[
	\begin{aligned}
		\mathbb E\int_{(0,T]}
		H^n_t\mathbf1_{\{t<\tau\}}\,da_t
		&=\frac1n\,
		\mathbb P(\tau>t_0)
		=\frac{e^{-\lambda t_0}}{n}.
	\end{aligned}
	\]
	
	Consequently,
	\[
	{
		\mathbb E\int_{(0,T]}H^n_t\,da_t
		=
		e^{\lambda t_0}
		\mathbb E\int_{(0,T]}
		H^n_t\mathbf1_{\{t<\tau\}}\,da_t.
	}
	\]
	
	Since $t_0\le T$, we also have
	\[
	{
		\mathbb E\int_{(0,T]}H^n_t\,da_t
		\le
		e^{\lambda T}
		\mathbb E\int_{(0,T]}
		H^n_t\mathbf1_{\{t<\tau\}}\,da_t \leq e^{\lambda T}
		\mathbb E\int_{(0,T]}
		H^n_t\mathbf\,da_t .
	}
	\]
	
	In particular, the stopped contact
	errors converge to zero if and only if
	the full contact errors converge to zero.
	
	This example shows that the quantitative
	support estimate remains valid when
	the reflection measure is singular
	with respect to Lebesgue measure
	and possesses an atom.
\end{example}

\section{Open problem: invariance-time reconstruction
	beyond Cox models}
\label{sec:invariance-open}

The Cox reconstruction theorem established above
relies on two particularly favorable features:
the immersion property and the explicit
compensator of the default indicator.
These properties allow the martingale,
finite-variation, and default-jump terms
to be identified separately.

A natural question is whether an analogous
reduction--reconstruction principle can be
established for the more general invariance
times introduced by Cr\'epey and
Song~\cite{CrepeySong}.

\medskip

\noindent
\textbf{The invariance-time framework.}
Recall that the framework of
Cr\'epey and Song involves two filtrations
$\mathbb F\subseteq\mathbb G$ and a
$\mathbb G$-stopping time $\tau$, together
with probability measures $\mathbb P$
and $\mathbb Q$, where $\mathbb P$ is
equivalent to the restriction of
$\mathbb Q$ to $\mathcal F_T$.

Under the corresponding reduction
hypotheses, $\tau$ is an invariance time
if $\mathbb P$ can be chosen so that
every $(\mathbb F,\mathbb P)$-local
martingale $N$, stopped immediately
before $\tau$, becomes a
$(\mathbb G,\mathbb Q)$-local martingale.

This property transports a class of
martingales across both filtration
enlargement and change of probability
measure. It does not, however,
automatically provide a complete
representation of the martingale terms
appearing in an enlarged-filtration
reflected BSDE.

In particular, the invariance property
does not by itself identify the default
jump, the compensator of the default
indicator, or the increasing processes
appearing in a Mertens reflection
decomposition.
These additional structures are
essential when the obstacle is
irregular and the reflection term
contains both predictable and
purely discontinuous optional
components.

\medskip

We therefore formulate the following
questions, which remain beyond the
scope of the preceding results.

\begin{enumerate}[label=\textbf{Q\arabic*.}]
	
	\item \emph{Martingale transport and
		default compensation.}
	
	Under what additional structural assumptions
	on the underlying filtrations can the
	invariance-time transformation be combined
	with a martingale representation theorem
	in the enlarged filtration?\footnote{
		A related problem has been studied
		by Crépéy~\cite{Crepey}.}
	
	More precisely, when can the
	martingale part of an enlarged
	reflected BSDE be decomposed into
	a transported reference-filtration
	martingale and a compensated
	default martingale, with a
	well-defined predictable
	compensator?
	
	\item \emph{Optional obstacle reduction.}
	
	Which regularity, measurability,
	and aggregation conditions ensure
	that an enlarged-filtration optional
	obstacle admits a compatible
	$\mathbb F$-optional reduction
	on the pre-default interval?
	
	In particular, how are its
	left-upper envelope and the
	right-jump structure of the
	corresponding reflected solution
	preserved under such a reduction?
	
	\item \emph{Mertens reconstruction.}
	
	Suppose the enlarged reflected
	BSDE involves the finite-variation
	reflection term
	\[
	A_t+C_{t-},
	\]
	where $A$ is predictable and
	increasing, while $C$ is optional,
	purely discontinuous, and increasing.
	
	Under which hypotheses can these
	two reflection components be reduced
	and reconstructed separately?
	
	How should possible jumps at
	$\tau$ be treated so that the
	reconstructed reflection term
	has the correct boundary behavior
	and creates no spurious mass
	at the random horizon?
	
	\item \emph{Existence and uniqueness.}
	
	Which conditions on the transformed
	driver, terminal condition,
	obstacle, and filtration allow
	the available existence and
	uniqueness results for reflected
	BSDEs with irregular obstacles
	to be applied in the reference
	filtration?
	
	Under which additional assumptions
	can the resulting solution be
	transported back to the enlarged
	filtration, yielding a genuine
	equivalence between the two
	reflected BSDE formulations?
	
\end{enumerate}

	The questions above concern the
	reconstruction of the entire reflected
	BSDE, rather than the transfer of
	its minimality conditions alone.
	
	Theorem~\ref{mertens} establishes
	componentwise transfer of predictable
	and optional contact identities
	under the appropriate support
	conditions.
	In the Cox setting, the positivity
	of the survival process yields
	particularly transparent criteria
	through Corollary~\ref{coxcor}.
	
	For a general invariance time,
	the relevant survival projections,
	reference measures, and probability
	measure changes must be handled
	consistently.
	Consequently, the contact-transfer
	results proved in this paper do not,
	by themselves, imply the existence,
	uniqueness, or reconstruction of
	solutions to irregular reflected
	BSDEs under arbitrary invariance
	times.

\section*{Conclusion}
\label{sec:conclusion}

This paper investigated the transfer of reflection
minimality conditions between a reference filtration
and an enlarged filtration under a random horizon,
with particular emphasis on the support of the
underlying reflection measures.

We established sharp necessary and sufficient
conditions for the equivalence of stopped and
unstopped contact identities through optional and
predictable dual projections. These criteria identify
when stopping at a random time preserves the
contact information carried by the reflection
measures, including when these measures possess
atomic or singular components.

For Cox random times, the explicit conditional
survival structure yields a strong transfer
principle. We extended the survival-weighting
identities to general reference-measurable random
measures and derived quantitative estimates under
uniform survival positivity. These estimates
provide a stability criterion for approximate
contact conditions.

In the Brownian Cox framework, we obtained an
explicit reconstruction of reflected BSDEs with
continuous obstacles. The reduced generator
incorporates the default intensity and the
prescribed recovery process, while the enlarged
solution is reconstructed through its Brownian,
default-jump, and reflection components.

For irregular obstacles, we extended the
contact-transfer analysis to the predictable
and purely discontinuous optional components
of Mertens reflection. In particular, we
clarified the relationship between integral
minimality and pointwise jump-contact conditions,
while respecting the distinct measurability
and endpoint conventions of the two components.

\paragraph{Scope of the results.}
The results of this paper establish sharp
support criteria for the transfer of optional
and predictable reflection-contact conditions,
together with their specialization to Cox
random times and quantitative stability under
uniform survival positivity.

In the Brownian Cox setting, the explicit
reconstruction theorem provides existence of
an enlarged-filtration solution and uniqueness
within the prescribed reference-filtration
reduction class. It does not establish
uniqueness among all possible
enlarged-filtration solutions.

For irregular obstacles, the results concern
the transfer of predictable and optional
minimality conditions under compatible
reductions; they do not establish a general
existence or reconstruction theorem.

Extensions to arbitrary invariance times
require additional assumptions concerning
martingale representation, obstacle reduction,
default jumps, and Mertens reflection,
as discussed in
Section~\ref{sec:invariance-open}.

The main contribution of the paper is therefore
to identify precisely which reflection-contact
properties can be transferred by measure-theoretic
support arguments, and which aspects of a complete
reflected BSDE reconstruction require additional
stochastic structure.

\end{document}